\RequirePackage{fix-cm}
\documentclass[smallextended]{svjour3}       
\smartqed  
\usepackage{amsmath}
\usepackage{amssymb}
\usepackage{mathrsfs,esint,bbm,stmaryrd,ulem}
\usepackage{enumitem,graphicx,xcolor}

\let\div\relax

\DeclareMathOperator{\div}{div}
\DeclareMathOperator{\Id}{Id}                                       
\DeclareMathOperator{\tr}{tr}                                       
\DeclareMathOperator{\dist}{dist}                                   
\DeclareMathOperator{\spt}{spt}                                     

\renewcommand{\d}{\mathrm{d}}
\newcommand{\N}{\mathbb{N}}       
\newcommand{\R}{\mathbb{R}}
\renewcommand{\S}{\mathbb{S}}
\newcommand{\n}{\mathbf{n}}       
\newcommand{\m}{\mathbf{m}}
\newcommand{\p}{\mathbf{p}}
\newcommand{\x}{\mathbf{x}}
\renewcommand{\v}{\mathbf{v}}
\renewcommand{\u}{\mathbf{u}}

\newcommand{\A}{\mathbf{A}}
\newcommand{\V}{\mathbf{V}}
\newcommand{\Q}{\mathbf{Q}}
\renewcommand{\P}{\mathbf{P}}
\newcommand{\NN}{\mathscr{Q}_{\mathrm{max}}}     
\renewcommand{\H}{\mathscr{H}}

\newcommand{\abs}[1]{\left|#1\right|}
\newcommand{\mint} {\mathop{\int\hskip -1,05em -\,}}
\newcommand{\mI}[1]{\mint\nolimits_{\!\!\!\!#1}}

\begin{document}

\title{Minimizers of a Landau-de Gennes Energy with a Subquadratic Elastic Energy
}


\author{Giacomo Canevari \and Apala Majumdar \and
        Bianca Stroffolini
}

\authorrunning{G. Canevari, A. Majumdar, \& B.  Stroffolini} 

\institute{G. Canevari \at
           Basque Center for Applied Mathematics, 
           Alameda de Mazarredo 14,
           48009 Bilbao (Spain).
           \email{gcanevari@bcamath.org}           
           \and
           A. Majumdar \at
           Mathematical Sciences, University of Bath, 
           Claverton Down, Bath, BA2 7A9 (United Kingdom).
           \email{a.majumdar@bath.ac.uk}
           \and
           B. Stroffolini \at
           Dipartimento di Matematica e Applicazioni,
           via Cintia,
           Universit\`a Federico II,
  80126, Napoli (Italy).
           \email{bstroffo@unina.it}}

\date{Received: date / Accepted: date}

\maketitle

\begin{center}
 \textit{\large{Dedicated to John M. Ball in the occasion of his 70th birthday.}}
\end{center}

\bigskip

\begin{abstract}
 We study a modified Landau-de Gennes model for nematic liquid crystals, where the elastic term
 is assumed to be of subquadratic growth in the gradient. We analyze the behaviour of global minimizers 
 in two- and three-dimensional domains, subject to uniaxial boundary conditions, 
 in the asymptotic regime where the length scale of the defect cores is small 
 compared to the length scale of the domain. We obtain uniform convergence of the minimizers
 and of their gradients, away from the singularities of the limiting uniaxial map.
 We also demonstrate the presence of maximally biaxial cores in minimizers on two-dimensional domains, 
 when the temperature is sufficiently low.
\end{abstract}

\section{Introduction}
\label{intro}

Liquid crystals (LCs) are classical examples of mesophases that combine the fluidity of liquids with the orientational and positional order of solids \cite{deGennes}. Nematic liquid crystals (NLCs) are the simplest type of LCs for which the constituent asymmetric molecules have no translational order but exhibit a degree of long-range orientational order i.e. certain distinguished directions of averaged molecular alignment in space and time. The mathematics of NLCs is very rich and there are at least three continuum theories for NLCs in the literature --- the Oseen-Frank, the Ericksen and the Landau-de Gennes theories. These theories typically have two key ingredients --- the concept of a macroscopic order parameter and a free energy whose minimizers model the physically observable stable nematic equilibria. The Oseen-Frank theory is the simplest continuum theory restricted to purely uniaxial nematics with a single preferred direction of molecular alignment and a constant degree of orientational order. The Oseen-Frank order parameter is just a unit-vector field that models this single special direction, with two degrees of freedom, referred to as the director field. The Oseen-Frank energy density is a quadratic function of the director and its spatial derivatives; in the so-called one-constant approximation, the Oseen-Frank energy density reduces to the Dirichlet energy density. The Oseen-Frank theory has been remarkably successful but is limited to purely uniaxial materials and can only describe low-dimensional defects. For example, minimizers of the Dirichlet energy density can only support isolated point defects and these point defects have the celebrated radial-hedgehog profile with the molecules pointing radially outwards everywhere from the point defect \cite{BrezisCoronLieb}. Minimizers of the Oseen-Frank free energy with multiple elastic constants (subject to certain constraints) have a defect set of Hausdorff dimension less than one \cite{HKL}. However, confined NLC systems frequently exhibit line defects and even surface defects or wall defects.

The Ericksen theory is also restricted to uniaxial nematics but can account for a variable degree of orientational order, labelled by an order parameter which vanishes at defect locations. This order parameter regularises higher-dimensional defects. The Landau-de Gennes (LdG) theory is the most powerful continuum theory for nematic liquid crystals and the LdG order parameter is the LdG $\mathbf{Q}$-tensor order parameter, which is mathematically speaking, a symmetric traceless $3\times 3$ matrix with five degrees of freedom. The LdG $\mathbf{Q}$-tensor can describe both uniaxial and biaxial nematic states, which have a primary and secondary direction of molecular alignment. The LdG free energy density usually comprises an elastic energy density (which is quadratic in the derivatives of the $\mathbf{Q}$-tensor) and a bulk potential which drives the isotropic-nematic phase transition induced by lowering the temperature and the exact relation between Oseen-Frank and LdG minimizers has received a lot of mathematical interest in recent years.

We do not give an exhaustive review here; one of the first rigorous results in this direction is an asymptotic result in the limit of vanishing elastic constant studied by Majumdar \& Zarnescu \cite{majumdarzarnescu} and subsequently refined by Nguyen \& Zarnescu \cite{NguyenZarnescu}. The authors study qualitative properties of LdG minimizers with a one-constant elastic energy density and show that the LdG minimizers for appropriately defined Dirichlet boundary-value problems on three-dimensional bounded simply-connected domains, converge strongly in $W^{1,2}$ to a limiting minimizing harmonic map, which is the minimizer of the one-constant Oseen-Frank energy. The limiting map has a discrete set of point defects and the LdG minimizers converge uniformly to the limiting map, everywhere away from the defects of the limiting map i.e. the limiting map is an excellent approximation of the LdG minimizers in this asymptotic limit, away from defects. 
In Contreras \& Lamy \cite{contreraslamy} and Henao, Majumdar \& Pisante~\cite{HenaoMajumdarPisante}, the authors study a different asymptotic limit, namely, the low temperature limit of minimizers of the LdG energy (with a one-constant elastic energy density) and prove that minimizers cannot have purely isotropic points with $\mathbf{Q}=0$ in this limit. Henao, Majumdar \& Pisante demonstrate the uniform convergence of LdG minimizers to a minimizing harmonic map, away from the singularities of the limiting map, in this asymptotic limit. Using topological arguments, Canevari \cite{Canevari2D} shows that the non-existence of isotropic points for suitably prescribed Dirichlet data implies the existence of points with maximal biaxiality and negative uniaxiality (uniaxial with negative order parameter) in global LdG minimizers in this limit. These results clearly illustrate two features: using the one-constant Dirichlet elastic energy densities, the Oseen-Frank minimizers provide excellent approximations to the LdG minimizers in certain asymptotic limits; the differences are primarily contained near the defect sets of the limiting maps and the defects of the limiting map and the LdG minimizers can have different structures. However, we would expect the LdG defects (for minimizers) to shrink to the defects of the limiting minimizing harmonic map in the limit of vanishing nematic correlation length. For example, it is well known from numerical simulations that LdG minimizers have biaxial tori as defect structures with a negatively ordered uniaxial defect loop and as the nematic correlation length shrinks (in the limit of vanishing elastic constant), the biaxial torus shrinks to the radial-hedgehog defect, which is the corresponding defect for the limiting Oseen-Frank minimizer. In this respect, defects of the limiting map do give some insight into the defects of the LdG minimizers and vice-versa. 

These continuum theories are variational theories with a quadratic elastic energy density or an energy density that is quadratic in the derivatives of the order parameter. However, there is little experimental evidence to support the quadratic behaviour in regions of large gradient i.e. near defects. Hence, it is reasonable to conjecture that the elastic energy density may be subquadratic near defects matched by a quadratic growth away from defects. For example, if the Oseen-Frank energy density was subquadratic for large values of the gradient, then line defects
would be captured by the Oseen-Frank theory. This would be a significant development since one of the most popular reasons for choosing the LdG theory over the Oseen-Frank theory are the limitations of the Oseen-Frank approach with respect to defects.

Building on this idea, we propose a variant of the LdG energy, with a modified elastic energy density and the LdG bulk potential, for Dirichlet boundary-value problems on three-dimensional domains. 
The modified elastic energy exhibits a subquadratic growth in $|\nabla \mathbf{Q}|^p$ with $1<p<2$,  for $|\nabla \mathbf{Q}|$ sufficiently large and interpolates to the usual Dirichlet energy density, $|\nabla \mathbf{Q}|^2$  for bounded values of the gradient.
This elastic energy density is necessarily not homogeneous, introducing various technical difficulties.
A suquadratic variant of the Oseen-Frank theory was proposed by Ball \& Bedford, \cite{BallBedford}.

We study minimizers of this modified LdG free energy, in the limit of vanishing elastic constant, by analogy with the work in Majumdar \& Zarnescu \cite{majumdarzarnescu}. The limiting map in our case, is a $\phi$-minimizing map with a defect set of zero $d-p$ Hausdorff measure, where $d=2$ or $d=3$ according to the dimension of the domain. The limiting map is $C^{1,\alpha}$ for $\alpha \in (0,1)$  away from the defect set and we prove that the modified LdG minimizers converge uniformly to the $\phi$-minimizing map away from the defect set of the $\phi$-minimizing map. The essential difference is that the $\phi$-minimizing map can support higher-dimensional defects, in contrast to the minimizing harmonic map which can only support point defects. As noted above, we would expect the LdG defects to converge to the defects of the $\phi$-minimizing map as the correlation length shrinks to zero and hence, a comprehensive study of the defects of the $\phi$-minimizing map can yield new possibilities for the modified LdG defects too.

The second part of our paper concerns a qualitative study of minimizers of the modified LdG free energy in the low-temperature limit, for two-dimensional domains. Our qualitative conclusions are the same as Contreras and Lamy \cite{contreraslamy}, who use the Dirichlet energy density i.e. the exclusion of purely isotropic points in global energy minimizers.

There are substantial technical differences between our work and previous work with the usual Dirichlet elastic energy density. The Euler-Lagrange equations in the modified case are only quasi-linear and not uniformly elliptic, we do not have exact monotonicity results for the normalized modified LdG energy on balls, we need different arguments for the regularity of the $\phi$-minimizing limiting map and in the low-temperature limit, we need more technical details since the limiting map is a $p$-minimizing harmonic map for $1<p<2$ as opposed to Contreras and Lamy who dealt with the $p=2$ case for the low-temperature limit of LdG minimizers on two dimensional domains.

Our strategy for regularity is the following: first we get Morrey and $C^{1.\alpha}$ estimates for the minimizers $\mathbf{Q}_L$ of the modified LdG functional, $I_{mod}$, which possibly 
depend on $L$. These estimates are needed in order to prove $L^{\infty}-L^1$ estimates for the gradients using the Bernstein-Uhlenbeck method of passing through an uniformly elliptic equation, see \cite{diestrover}.
The final goal of uniform estimates is reached by a combined use of monotonicity of the energy and a scaling procedure that does not affect the characteristics of $\phi$, see Proposition~\ref {prop:uniform1}.

From a physical standpoint, the overall story for a modified LdG elastic energy density that is subquadratic in $|\nabla \mathbf{Q}|$ in regions of large gradient, seems to be similar to the story for a Dirichlet elastic energy density with the difference being captured by the limiting $\phi$-minimizing map as compared to the limiting minimizing harmonic map. The $\phi$-minimizing map is expected to have a more complicated and higher-dimensional defect set and this will have consequences for the LdG minimizers too. However, it remains a difficult task to test these theoretical predictions for defect structures since the experimental resolution of defect structures or the determination of the elastic energy density near defects are open issues.

\section{Setting of the problem and statement of the main result}
\label{sec:1}

Let $\Omega\subset\R^d$ be a bounded, smooth domain of dimension~$d\in\{2, \, 3\}$.
Let $S_0$ denote the space of  symmetric, traceless $3\times 3$ matrices given by
\begin{equation}
\label{eq:1}
S_0 = \left\{ \mathbf{Q}\in M^{3\times 3}\colon Q_{ij} =Q_{ji}; \ Q_{ii} = 0 \right\}
\end{equation}
where $M^{3\times 3}$ is the set of $3\times 3$ matrices, 
$\mathbf{Q} = \left(Q_{ij} \right)$ and we have used Einstein summation convention. 
The matrix $\mathbf{Q}$ is the Landau-de Gennes tensor parameter. 
In particular, (i) $\mathbf{Q}$​ is biaxial if it has three distinct eigenvalues; 
(ii) uniaxial if it has two non-zero degenerate eigenvalues such that the eigenvector 
associated with the non-degenerate eigenvalue is the distinguished director 
and (iii) isotropic if $\mathbf{Q}=0$.
We study minimizers of the modified Landau-de Gennes energy functional
\begin{equation} \label{eq:2old}
 \begin{split}
  I_{mod}[\mathbf{Q}] &= \int_{\Omega} \phi\left(|\nabla \mathbf{Q}|\right) 
      + \frac{1}{L}f_B\left(\mathbf{Q}\right) dV \\
      &= \int_{\Omega} \psi\left(|\nabla \mathbf{Q} |^2 \right) 
      + \frac{1}{L}f_B\left(\mathbf{Q}\right) dV
 \end{split}
\end{equation}
where $\psi(t^2):=\phi(t)$.
We will assume the following on $\phi$:
 \begin{enumerate}[label=(H\textsubscript{\arabic*}), ref=H\textsubscript{\arabic*}]
  \item \label{hp:C2} \label{hp:first} $\phi\in C^1[0, \, \infty)\cap C^2(0, \, \infty)$.
  \item \label{hp:convex} $\phi(0) = \phi^\prime(0) = 0$ and~$\phi^{\prime\prime}(t)>0$ for any~$t>0$. 

  \item \label{hp:Delta2prime} There exists positive numbers $c_0 \leq c_1$ such that
  \[
   c_0\phi^\prime(t) \leq \phi^{\prime\prime}(t)t
   \leq c_1 \phi^\prime(t) \qquad \textrm{for any } t\geq 0.
  \]
  \item \label{hp:subquadratic} There exists a number~$p\in (1, \, 2)$ such that 
  $\sup_{t\geq 0}(\phi^{\prime}(t)t - p\phi(t)) <+\infty$. 
  
  \item \label{hp:holder}  \label{hp:last} $\phi^{\prime\prime}$ is H\"older continuous off the diagonal:
    \[ \left|\phi^{\prime\prime}(s+t)-\phi^{\prime\prime}(t)\right|\leq
    c\, \phi^{\prime\prime}(t)\, 
    \left(\frac{\left|s\right|}{t} \right)^\beta, \quad \beta>0 \]
  for all $t>0$ and $s \in\mathbb{R}$ with $|s| < t/2$.
 \end{enumerate}
 An example of admissible~$\phi$ is
 \begin{equation} \label{specialphi}
  \phi(t) = \frac{k^{1-p/2}}{p}\left(t^2 + k\right)^{p/2} - \frac{k}{p}
 \end{equation}
 where~$k>0$ is a fixed parameter (compare with~\cite{BallBedford}).
Notice that the assumptions \eqref{hp:C2}--\eqref{hp:Delta2prime}, \eqref{hp:holder} guarantee the excess decay estimate for local minimizers of functional of Uhlenbeck type with $\phi$-growth, see \cite{diestrover}.

Further, $f_B$ is the usual quartic thermotropic potential that dictates the
isotropic-nematic phase transition as a function of the temperature \cite{ejam2010,ballnotes}:
\begin{equation}
f_B\left(\mathbf{Q}\right): = -\frac{A}{2}\textrm{tr}\mathbf{Q}^2 - \frac{B}{3}\textrm{tr}\mathbf{Q}^3 + \frac{C}{4}\left(\textrm{tr}\mathbf{Q}^2\right)^2 +M\left(A, B, C \right)
\end{equation} where $\textrm{tr}\mathbf{Q}^n = \sum_{i=1}^{3} \lambda_i^n $ for $n\geq 1$, $\sum_{i=1}^{3}\lambda_i = 0$, $A$ is the re-scaled temperature and $B, C$ are positive material-dependent constants whilst $L>0$ in (\ref{eq:2old}) is a fixed material-dependent elastic constant. The bulk potential $f_B$ is bounded from below and we add the constant $M(A,B,C)$ to ensure that 
$\min_{S_0} f_B = 0$. We work with temperatures below the critical nematic supercooling temperature or roughly speaking, we work with low temperatures so that $A>0$ in this  paper and $f_B$ attains its minimum 
on the set of uniaxial $\mathbf{Q}$-tensors given below:

\begin{equation}
\label{eq:Qmin}
\NN = \left\{ \mathbf{Q}\in S_0\colon \mathbf{Q} = s_+\left(\mathbf{n}\otimes\mathbf{n} - \mathbf{I}/3 \right)\right\}
\end{equation}
with $s_+ = \frac{B + \sqrt{B^2 + 24 A C}}{4C}$ and $\n\in S^2$ an arbitrary unit-vector \cite{newtonmottram,ejam2010,ballnotes}.

We take our admissible space to be
\begin{equation}
\label{eq:3}
\mathcal{A} = \left\{ \mathbf{Q}\in W^{1,\phi}(\Omega; S_0)\colon f_B(\Q)\in L^1(\Omega), \ \Q = \Q_b ~\textrm{on }\partial \Omega\right\}
\end{equation}
where $W^{1,\phi}(\Omega; S_0)$ is the Orlicz-Sobolev space of $L^{\phi}$-integrable $\mathbf{Q}$-tensors with $\nabla \mathbf{Q}\in L^{\phi}(\Omega)$, see Section~\ref{sect:notation} . 
The Dirichlet boundary condition $\mathbf{Q}_b \in W^{1,\phi}(\Omega; \NN)$
by assumption, since this is a physically relevant choice that simplifies
the subsequent analysis. In other words, we assume that
\begin{equation}
\label{eq:Qb}
\Q_b = s_+( \n_b \otimes \n_b - I/3 )
\end{equation}
where $I$ is the $3\times 3$ identity matrix, $\n_b\colon\Omega\to S^2$ and $\n_b \otimes \n_b \in W^{1,\phi}(\Omega; M^{3\times 3})$.

\begin{remark} \label{remark:bd_datum}

We have assumed that the boundary condition~$\mathbf{Q}_b$ is actually 
 defined on the whole of the domain~$\Omega$, and belongs to the Sobolev-Orlicz space $W^{1,\phi}$.
 However, in practical applications the behaviour of $\mathbf{Q}$ may only be assigned
 on the boundary of~$\Omega$, and one might ask whether there exists a
 map~$\mathbf{Q}_b \in W^{1,\phi}(\Omega; \NN)$ that matches the prescribed behaviour
 at the boundary. A sufficient condition for the existence of such~$\Q_b$ is the following:
 let~$p\in (1, \, 2)$ be given by Assumption~\eqref{hp:subquadratic}, and let 
 $\mathbf{P}\in W^{1-1/p, p}(\partial\Omega; \NN)$ be given; then, there exists a map
 $\mathbf{Q}_b\in W^{1, p}(\Omega; \NN)$ such that $\Q_b = \mathbf{P}$ on~$\partial\Omega$,
 in the sense of traces \cite[Theorem~6.2]{HardtLin}. The assumption~\eqref{hp:subquadratic}
 implies that $\phi(t)\lesssim t^p + 1$ and hence,
 we also have $\mathbf{Q}_b\in W^{1,\phi}(\Omega; \NN)$. Such an 
 extension~$\Q_b\in W^{1,p}(\Omega; \NN)$ might \emph{not} exist in case~$p = 2$,
 due to topological obstructions associated with the manifold~$\NN$ 
 (see e.g. \cite[Proposition~6]{Canevari3D}).
\end{remark}

In what follows, we re-scale the energy (\ref{eq:2old}); let $\bar{\mathbf{x}} = \frac{\mathbf{x}}{D}$ where $D$ is a characteristic length scale of the domain $\Omega$. It is a straightforward exercise to show that the re-scaled energy is
\begin{equation} \label{eq:2}
 \begin{split}
  \bar{I}_{mod}[\mathbf{Q}] &= \int_{\bar{\Omega}} D^3 \phi\left(\frac{|\bar{\nabla} \mathbf{Q}|}{D}\right) 
      + \frac{D^3}{L}f_B\left(\mathbf{Q}\right) \bar{dV} \\
      &= D \int_{\bar{\Omega}} \bar{\psi}\left(|\bar{\nabla} \mathbf{Q} |^2 \right) 
      + \frac{1}{\bar{L}}\bar{f_B}\left(\mathbf{Q}\right) \bar{dV}
 \end{split}
\end{equation}
where $\bar{\psi} = D^2\psi\left(\frac{|\bar{\nabla} \mathbf{Q} |^2}{D^2} \right) $, $\bar{f_B} = f_B*\frac{1}{A_0}$, $A_0>0$ is some characteristic value of the temperature variable $A$ and $\bar{L} = \frac{L}{ A_0 D^2}$. In what follows, we will work with the re-scaled energy (\ref{eq:2}) and drop the bars for brevity. In particular, we will study qualitative properties of minimizers of (\ref{eq:2}) in the limit $\bar{L}\to 0$ which is the macroscopic limit that describes $D^2 \gg \frac{L}{A_0} $, for a typical correlation length $\xi\propto \sqrt{\frac{L}{A_0}}$.
To this purpose, we define a $\phi$-minimizing uniaxial tensor-valued harmonic map,
by analogy with the ``minimizing harmonic map''
employed for the Dirichlet elastic energy density i.e. $|\nabla \mathbf{Q}|^2$ in \cite{majumdarzarnescu}.
All subsequent results and statements are to be interpreted in terms of the re-scaled energy (\ref{eq:2}).
\begin{definition} \label{def:harmonic}
A $\phi$-minimizing uniaxial harmonic map is a minimizer~$\Q_0\in W^{1,\phi}(\Omega;\NN)$ 
of the functional
\begin{equation*}
 \Q\mapsto \int_\Omega \phi(|\nabla\Q|) dV
\end{equation*}
among all maps~$\Q\in W^{1,\phi}(\Omega;\NN)$ such that $\Q = \Q_b$ on~$\partial\Omega$.
Equivalently, a $\phi$-minimizing uniaxial harmonic map is given by
\begin{equation} \label{eq:Q0}
\Q_0 = s_+ \left(\n_0 \otimes \n_0 - I/3 \right)
\end{equation}
for a unit-vector field $\n_0:\Omega \to S^2$ such that the symmetric matrix $\n_0\otimes \n_0$ 
is a global minimizer of the functional
\begin{equation}
\int_{\Omega} \psi(s_+^2|\nabla (\n_0 \otimes \n_0)|^2) dV = 
\min_{\n \otimes \n \in \mathcal{A}_\n} \int_{\Omega}\psi(s_+^2|\nabla(\n\otimes \n)|^2)~dV
\end{equation} 
in the admissible space
\[
\mathcal{A}_\n = \left\{\n: \Omega \to S^2\colon \n\otimes \n \in W^{1,\phi}\left(\Omega; M^{3\times 3}\right);~
\n\otimes\n = \n_b\otimes\n_b ~\textrm{on } \partial \Omega \right\}
\]
and $\Q_b$ and $\n_b$ are related as in (\ref{eq:Qb}).
\end{definition}

We can now state our main result.

\begin{theorem} \label{th:convergence}
 Suppose that the elastic energy density~$\phi$ satisfies the
 Assumptions~\eqref{hp:first}--\eqref{hp:last} above.
 Let~$\Q_L$ be a minimizer of the functional~\eqref{eq:2} 
 in the admissible class~$\mathcal{A}$ defined by~\eqref{eq:3}.
 Then, there exists a subsequence $L_k\to 0$ as~$k\to+\infty$
 and a $\phi$-minimizing uniaxial
 harmonic map~$\Q_0$ such that the following properties hold:
 \begin{enumerate}[label=(\roman*)]
  \item the set 
  \[
   S[\Q_0] := \left\{\x\in\Omega\colon 
   \liminf_{\rho\to 0} \rho^{p-d} \int_{B(\x, \, \rho)} \phi(|\nabla\Q_0|) >0 \right\} \!,
  \]
  where~$p\in (1, \, 2)$ is given by Assumption~\eqref{hp:subquadratic},
  is closed and there holds $\H^{d-p}(S[\Q_0]) = 0$;
  
  \item $\Q_0\in C^{1, \alpha}_{\mathrm{loc}}(\Omega\setminus S[\Q_0])$
  for some~$\alpha\in (0, \, 1)$;
  
  \item we have $\Q_{L_k}\to \Q_0$, $\nabla\Q_{L_k}\to\nabla\Q_0$
  locally uniformly on~$\Omega\setminus S[\Q_0]$ as~$k\to+\infty$.
 \end{enumerate}
\end{theorem}

\subsection{Notation, Orlicz spaces}
\label{sect:notation}

In what follows, we use the notations $f\sim g$ and~$f\lesssim g$
as short-hand for $c_0f\leq g\leq c_1 f$ and~$f \leq c_2 g$ respectively,
$c_0$, $c_1$, $c_2$ being some positive constants. 
We recall here some standard facts about N-functions 
(see e.g. \cite{RaoRen} for more details).

A real function $\phi\colon [0, \, +\infty) \to [0, \, +\infty)$
is said to be an N-function if $\phi(0)=0$, $\phi$ is differentiable, the derivative~$\phi^\prime$
is right continuous, non-decreasing and satisfies $\phi^\prime(0) = 0$
and $\phi^\prime(t) >0$ for $t > 0$. In particular, an N-function is convex. We say that 
$\phi$ satisfies the $\Delta_2$-condition if there exists $c > 0$ such that
$\phi(2t) \leq c\phi(t)$ for any~$t\geq 0$. We denote by $\Delta_2(\phi)$ the smallest constant~$c$
such that the previous inequality holds.
Given two N-functions $\phi_1$, $\phi_2$, we define $\Delta_2(\phi_1, \phi_2) := \max_{i=1,2}\Delta_2(\phi_i)$.
If $\phi$ is an N-function that satisfies the $\Delta_2$-condition, then
\begin{equation} \label{subadd}
 \phi(t+s) \leq c \, \phi\left(\frac{t+s}{2}\right) \leq
 \frac{c}{2}\left( \phi(t) + \phi(s) \right)
 \qquad \textrm{for all } s, \, t\geq 0.
\end{equation}

If~$\phi^\prime$ is strictly increasing, then we denote 
by $(\phi^\prime)^{-1} \colon [0, \, +\infty) \to [0, \, +\infty)$
the inverse function of~$\phi$, and we define
\begin{equation} \label{phi*}
 \phi^*(t) := \int_0^t (\phi^\prime)^{-1}(s) \, \d s \qquad \textrm{for any } t\geq 0.
\end{equation}
The function~$\phi^*$ is called the Young-Fenchel-Yosida dual function of~$\phi$. The functions~$\phi$
and~$\phi^*$ satisfy the so-called Young inequality, namely,
for any~$\epsilon>0$ there is $C_\epsilon>0$ such that
\begin{equation} \label{Young}
 st \leq \epsilon\phi(s) + C_\epsilon\phi^*(t) \qquad \textrm{for any } s, \, t\geq 0.
\end{equation}
If~$\epsilon = 1$, then we can take~$C_\epsilon=1$.

We can restate \eqref{hp:Delta2prime} in this way:
  \begin{align}
    \label{eq:phi_pp}
    \phi'(t) &\sim t\,\phi''(t)
  \end{align}
  uniformly in $t > 0$. The constants in~\eqref{eq:phi_pp} are called
  the {\em characteristics of~$\phi$}.
We remark that under these assumptions $\Delta_2({\phi,\phi^\ast})
< \infty$ will be automatically satisfied, where
$\Delta_2({\phi,\phi^*})$ depends only on the characteristics
of~$\phi$ and $\phi^\ast$.

Next, we define the Orlicz space $L^{\phi}(\Omega)$ as the space of measurable function $u$ such that 
$\int_{\Omega} \phi(|u(x)|) dx <\infty.$ The Orlicz space is a Banach space, also it is reflexive if the function $\phi$ verifies the $\Delta_2$ condition and its dual is the space $L^{\phi^*}$.
The Orlicz-Sobolev space $W^{1, \phi}(\Omega)$ is defined accordingly by requiring that both $u$ and the distributional gradient $\nabla u$ belong to $L^{\phi}$.

For a given N-function $\phi$, we define the N-function $\omega$ by
\begin{align}
  \label{eq:def_psi}
  \omega'(t) &:= \sqrt{ \phi'(t)\,t\,}.
\end{align}
We remark that if $\phi$ satisfies the condition~\eqref{eq:phi_pp}, 
then also $\phi^*$, $\omega$, and $\omega^*$ satisfy this condition.

Define $\A,\V\colon\R^d\otimes S_0\to\R^d\otimes S_0$ in the following way:
\begin{subequations}
  \label{eq:defAV}
  \begin{align}
    \label{eq:defA}
    \A(\mathbf{D})&=\phi'(|\mathbf{D}|)\frac{\mathbf{D}}{|\mathbf{D}|},
    \\
    \label{eq:defV}
    \V(\mathbf{D})&=\omega'(|\mathbf{D}|)\frac{\mathbf{D}}{|\mathbf{D}|}.
  \end{align}
\end{subequations}
The function $\A$ represents the leading term of the $\phi$-Laplacian system, while the function $\V$ , called the ``excess'' function, is the nonlinear expression for the excess decay, see Theorem \ref{generalgrowth}.

Another important set of tools are the {\rm shifted N-functions}
$\{\phi_a \}_{a \ge 0}$. We define for $t\geq0$
\begin{align}
  \label{eq:phi_shifted}
  \phi_a(t):= \int _0^t \phi_a'(s)\, \d s\qquad\text{with }\quad
  \phi'_a(t):=\phi'(a+t)\frac {t}{a+t}.
\end{align}
Note that $\phi_a(t) \sim \phi'_a(t)\,t$. The families $\{\phi_a \}_{a \ge 0}$ and
$\{(\phi_a)^* \}_{a \ge 0}$ satisfy the $\Delta_2$-condition uniformly in $a \ge 0$. 
The connection between $\A$, $\V$ (see {\cite{DieningEttwein}})  is the following:
\begin{align}\label{eq:equivalence}
    \big({\A}(\mathbf{D}_1) - {\A}(\mathbf{D}_2)\big) \cdot \big(\mathbf{D}_1-\mathbf{D}_2 \big)
    &\sim |{ \V(\mathbf{D}_1) - \V(\mathbf{D}_2)}|^2 \sim
    \phi_{|{\mathbf{D}_1}|}(|{\mathbf{D}_1 - \mathbf{D}_2}|),
\end{align}
uniformly in $\mathbf{D}_1$, $\mathbf{D}_2$. Moreover,
\begin{align}
   \A(\mathbf{D}) \cdot \mathbf{D} \sim |{\V(\mathbf{D})}|^2 &\sim \phi(|{\mathbf{D}}|),
\end{align}
uniformly in $\mathbf{D}$.

\begin{lemma}\label{lem:Orliczproperties}
Suppose $\phi\colon[0, \, +\infty)\to[0, \, +\infty)$ is an $N$-function 
and that $\phi$, $\phi^*$ both satisfy the $\Delta_{2}$-condition. Then we have, 
uniformly in $\lambda\in [0,1]$ and $a\geq 0$, 
\begin{align*}
 \phi_{a}(\lambda a)\sim \lambda^{2}\phi(a) \qquad \textrm{and} \qquad
 \phi_{a}^{*}(\lambda\phi'(a))\sim\lambda^{2}\phi(a). 
\end{align*}
\end{lemma}
The proof of this lemma is a starightforward computation, based on the definition
of the shifted function~\eqref{eq:phi_shifted}.


In the paper \cite{diestrover}, the authors proven that the analogue
of Uhlenbeck result holds true for functionals with general growth.

\begin{theorem}[\cite{diestrover}] \label{generalgrowth}
Let  $\phi\in C^1([0, \infty))\cap C^2 ((0, \infty))$ be a convex function such that
 \begin{itemize}
   \item {G1.} {$\phi'(t)\sim t \phi''(t)$} uniformly in $t>0$
    \item G2.  $\phi''$ is H\"older continuous off the diagonal:
    {$$ \left|\phi''(s+t)-\phi''(t)\right|\leq c\, \phi''
(t)\, \bigg(
    \frac{\left|s\right|}{t} \bigg)^\beta \quad \beta>0 $$}
  for all $t>0$ and $s \in \mathbb{R}$ with $\left|s\right| < \frac{1}{2} t$.
    \end{itemize}
There exists a constant $c\geq 1$ and an exponent $\gamma \in (0,1)$ depending only on $n,N$ and the characteristics of $\phi$ such that the following statement holds true: whenever $h \in W^{1,\phi}_0(B_{R}(x_0),\mathbb{R}^{N})$ is a weak solution of the system
\begin{equation*}
 \div\left(\frac{\phi'(\abs{\nabla u})}{|\nabla u|} \, \nabla u\right) =  0
 \qquad \textrm{in } B_{R}(x_0) , 
\end{equation*}
then for every $0 < r < R$, we have
\begin{equation}\label{excess}
 \begin{split}
  \sup_{B_{R/2}(x_0)} \phi(\abs{\nabla h}) \leq c \mI{B_{R}(x_0)} \phi(\abs{\nabla h})\, \d x 
  \qquad \text{and} \qquad \\
  \Phi(h;x_0,r,(\nabla h)_{x_0,r}) \, \leq \, c \, \Big( \frac{r}{R} \Big)^{2 \gamma} \, 
  \Phi(h;x_0,R,(\nabla h)_{x_0,R}) ,
 \end{split}
\end{equation}
where $\Phi(h;x_0,r,(\nabla h)_{x_0,r})$ is defined through the function $\V$:
\begin{equation*}
\Phi(h;x_0,r,(\nabla h)_{x_0,r}):=\mI{B(x_0,r)} \abs{\V(\nabla h)-(\V(\nabla h))_{x_0,r}}^2 \, \d x
\end{equation*}
\end{theorem}

\section{Asymptotic analysis of minimizers}

\subsection{Preliminaries}

\begin{lemma} \label{lemma:existence}
For every $A>0$, the functional $I_{mod}$ defined in \eqref{eq:2} admits a global minimizer $\mathbf{Q}_A \in \mathcal{A}$.
\end{lemma}
\begin{proof}
The proof follows immediately from the direct methods in the calculus of variations. The admissible space $\mathcal{A}$ is non-empty since the Dirichlet boundary condition $\mathbf{Q}_b \in \mathcal{A}$ for each $A>0$. The energy density in (\ref{eq:2}) is bounded from below and the energy density is a convex function of $\nabla \mathbf{Q}$ for each $A>0$, hence the functional $I_{mod}$ in (\ref{eq:2}) is bounded from below, $W^{1,\phi}$-coercive and lower semi-continuous for each $A>0$ \cite{evans}. 
Furthermore, the set $\mathcal{A}$ is weakly closed. 
This is enough to guarantee the existence of a global energy minimizer $\mathbf{Q}_A \in \mathcal{A}$.
\end{proof}

The critical points of $I_{mod}$ in \eqref{eq:2} are $C^{1,\alpha}$-solutions (for some $0<\alpha<1$)
of the following system of Euler-Lagrange equations:
\begin{equation}
\label{eq:EL}
\frac{\partial}{\partial x_k}\left[ 2 \psi^\prime(|\nabla\Q|^2) Q_{ij,k}\right] = \frac{1}{L}\left( -A Q_{ij} - B\left( Q_{ip}Q_{pj} - |\mathbf{Q}|^2 \delta_{ij}/3 \right) + C Q_{pq}Q_{pq} Q_{ij} \right) \quad
\end{equation}
for $ i,j,k=1,2,3$, where $\psi^\prime = \frac{\d \psi}{\d t}$, $Q_{ij,k} = \frac{\partial Q_{ij}}{\partial x_k}$ and 
$x_k$ is the $k$-th component of the position vector $\x\in \Omega$.

\begin{equation}\label{eq:EL2}
\div\left(\psi^\prime(|\nabla\Q|) 2 \nabla\Q\right)=\frac{1}{L}\left( -A Q_{ij} - B\left( Q_{ip}Q_{pj} - |\mathbf{Q}|^2 \delta_{ij}/3 \right) + C Q_{pq}Q_{pq} Q_{ij} \right) 
\end{equation}
for $ i,j,k=1,2,3$.
The $C^{1,\alpha}$-regularity of the weak solutions of the system \eqref{eq:EL} was first proven in the $p$-Laplacian case for $p>2$ \cite{uhlenbeck}, for $1<p<2$ see \cite{acerbifusco}, for convex functions of general growth, see  \cite{marcellinipapi}, and   \cite{diestrover} where there is an excess decay. 
See Section~\ref{sect:splitting} below for problems with a right-hand side.
In \cite{changyouwang}, the author uses the $C^{1,\alpha}$-regularity of solutions of a $p$-Ginzburg Landau-type system to deduce qualitative properties of minimizers of a $p$-Ginzburg Landau functional.

Our first result is a maximum principle argument for all solutions of the system \eqref{eq:EL}.
\begin{lemma} \label{lemma:max_principle}
Let $\mathbf{Q}_c \in \mathcal{A}$ be a critical point of $I_{mod}$ in \eqref{eq:2}. Then
$$ |\mathbf{Q}_c|^2 \leq \frac{2}{3}s_+^2 $$
on $\bar{\Omega}$.
\end{lemma}

\begin{proof}
The Euler-Lagrange equations \eqref{eq:EL} can be written as
\begin{align*}
&2 \psi^\prime(|\nabla\Q|^2) Q_{ij,kk} + 4 Q_{ij,k} \psi^{\prime\prime}(|\nabla\Q|^2) Q_{pq,r} Q_{pq,rk} \\
&\qquad\qquad = \frac{1}{L}\left( - A Q_{ij} - B\left(Q_{ip} Q_{pj} - |\mathbf{Q}|^2 \delta_{ij}/3 \right) + C|\mathbf{Q}|^2 Q_{ij} \right).
\end{align*} 
We assume that $|\mathbf{Q}|^2$ attains a maximum at an interior point $\x_0 \in \Omega$. Note that 
if $|\mathbf{Q}|^2$ attains a maximum at an interior point $\x_0 \in \Omega$, then $Q_{ij} Q_{ij,k} = 0$ at $\x_0$.
We multiply both sides of the above equation with $Q_{ij}$ using Einstein summation convention, and obtain the following equality at the interior maximum point ${\x}_0 \in \Omega$:
\begin{equation}\label{eq:max}
\psi^\prime(|\nabla\Q|^2) \Delta(|\Q|^2) = 2 \psi^\prime(|\nabla\Q|^2) |\nabla \mathbf{Q}|^2 + g\left(\mathbf{Q}\right) ~ \textrm{at } \x_0 \in \Omega,
\end{equation}
where $g\left(\mathbf{Q}\right) = \frac{1}{L}\left( - A |\mathbf{Q}|^2 - B\textrm{tr}\mathbf{Q}^3 + C |\mathbf{Q}|^4 \right)$.
In \cite{ejam2010}, it is explicitly shown that $g\left( \mathbf{Q}\right) >0$ for
$|\mathbf{Q}|^2 > \frac{2}{3} s_+^2$.
Using the definition of $\psi(t)$ and an interior maximum point, we have that the left-hand side of (\ref{eq:max}) is non-positive whereas the right-hand side is strictly positive if $|\mathbf{Q}|\left( \x_0 \right)^2 > \frac{2}{3} s_+^2$.
Therefore, we must have
$$ |\mathbf{Q}|^2\left(\x_0 \right) \leq \frac{2}{3} s_+^2$$
at an interior maximum point $\x_0 \in \Omega$. Since the boundary datum~$\Q_b$ defined in~\eqref{eq:Qb} 
satisfies $|\Q_b|^2\leq \frac{2}{3} s_+^2$, the conclusion of the lemma follows. 
\end{proof}

Next, we consider the rescaled energy 
on balls $B(\x, r) \subset \Omega$ defined to be
\begin{equation} \label{eq:n1}
 \begin{split}
  \mathcal{F}_L\left(\mathbf{Q}, \x,r \right)
    &:= \frac{1}{r^{d-p}} \int_{B\left(\x, r \right)} e_L \left(\mathbf{Q}\right) dV \\
    &= \frac{1}{r^{d-p}}\int_{B(\x, r)} \phi(|\nabla \mathbf{Q}|) 
    + \frac{1}{L} f_B\left(\mathbf{Q}\right)~dV
 \end{split}
\end{equation} 
where the number~$p\in(1, \, 2)$, depending on~$\phi$, is given by Assumption~\eqref{hp:subquadratic}.
In~\cite{majumdarzarnescu}, the authors consider the usual Dirichlet energy density,
$\phi(|\nabla \mathbf{Q}|) = \frac{1}{2}|\nabla \mathbf{Q}|^2$, and show that the rescaled
energy is an increasing function of the ball radius $r$.  In our setting, 
we can show instead an ``almost-monotonicity'' formula, i.e. we show that the rescaled energy 
as defined above, is an increasing function of $r$ ``up to an error'' that we can bound.

\begin{lemma} \label{lem:mon}
Let $\mathbf{Q}_L\in \mathcal{A}$ be a minimizer of the functional~$I_{mod}$, 
and let
\begin{equation} \label{monotonicity_error}
 \xi(t) := \phi^\prime(t)t - p \phi(t) \qquad \textrm{for any } t\geq 0.
\end{equation}
Then, for any~$\x_0\in\Omega$ and any $0 < r \leq R$ such that 
$B(\x_0, R)\subset\Omega$, there holds
\begin{equation} \label{eq:G1}
 \begin{split}
  \mathcal{F}_L(\mathbf{Q}_L,\x_0, r) &\leq \mathcal{F}_L(\mathbf{Q}_L, \x_0, R)
  + \int_{r}^R\left(\rho^{p-d-1}\int_{B(\x_0, \rho)} \xi(|\nabla\Q|)~dV\right) \, \d\rho \\
  &\leq \mathcal{F}_L(\mathbf{Q}_L, \x_0, R) + M\left(R^p - r^p\right)
 \end{split}
\end{equation}
 for some constant~$M>0$ that only depends on~$d$,~$\phi$ and~$p$.
\end{lemma}
Since~$\phi(t)\sim t^2$ for small~$t$, we have~$\phi(t) > 0$ 
for small~$t>0$, whence the extra term. However, that term is bounded 
and small if $R$ is small, so it is not a problem (see proof of Prop.~\ref{prop:uniform1}).
\begin{proof}
  The monotonicity formula has been shown in the quadratic case \cite[Lemma $2$]{majumdarzarnescu} 
 by multiplying both sides of the Euler-Lagrange equation by $x_k\partial_k Q_{ij}$ and integrating by parts.
 This argument could be adapted to our case but, in order to make the computation rigourous, 
 some control on the second derivatives of~$\Q_L$ is needed. In order to avoid this technicality, 
 we adopt here another approach. For simplicity, we assume, $\x_0 = 0$, we write~$\Q$ 
 instead of~$\Q_L$  and~$B_\rho$ instead of~$B(\mathbf{0}, \rho)$.
 
 We first note that 
 \begin{equation} \label{derivative_rescaleden}
  \frac{\d}{\d \rho} \left\{{\rho}^{p-d} \int_{B_{\rho}} e_L(\Q)\right\} 
  = (p-d) {\rho}^{p-d-1}\int_{B_{\rho}} e_L(\Q) + \int_{\partial B_{\rho}} e_L(\Q).
 \end{equation}
 Now, we consider the so-called ``inner variation'' of~$\Q$, that is, we take 
 a regular vector field $\mathbf{X}\in C^{\infty}_{\mathrm{c}}(\Omega, \mathbb{R}^d)$ and,
 for~$t\in\R$, we define the maps $\mathbf{X}^t := \Id+t\mathbf{X}$. For~$|t|$ small enough,
 $\mathbf{X}^t$ maps~$\Omega$ diffeomorphically into itself and so it makes sense to define 
 $\Q^t := \Q\circ\mathbf{X}^t$ (here the~$\circ$ denotes the composition of maps).
 We can now compute that
 \[
  |\nabla\Q^t|^2 = |\nabla\Q|^2\circ\mathbf{X}^t + 2t
  (\partial_k Q_{ij}\circ\mathbf{X}^t)(\partial_p Q_{ij}\circ\mathbf{X}^t)\partial_k X_p
 \]
 and hence,
 \[
  \begin{split}
   &\phi(|\nabla\Q^t|) = \phi(|\nabla\Q|)\circ\mathbf{X}^t  \\ 
   &\qquad + t\frac{\phi^\prime(|\nabla\Q|)\circ\mathbf{X}^t}{{|\nabla\Q|}\circ\mathbf{X}^t}
     (\partial_k Q_{ij}\circ\mathbf{X}^t)(\partial_p Q_{ij}\circ\mathbf{X}^t)\partial_k X_p
     + \mathrm{o}(t)
  \end{split}
 \]
 On the other hand, by the implicit function theorem we can express the inverse of
 $\mathbf{X}^t$ as~$(\mathbf{X}^t)^{-1} = \Id - t\mathbf{X} + \mathrm{o}(t)$.
 Therefore, by using the identity $\det(\Id + tA) = 1 + t \tr A + \mathrm{o}(t)$, we obtain
 \[
  \det\nabla((\mathbf{X}^t)^{-1}) = \det(\Id - t\nabla\mathbf{X}) + \mathrm{o}(t) 
  = 1 - t\div\mathbf{X} + \mathrm{o}(t)
 \]
 Using the previous information, we can evaluate the energy of~$\Q^t$ by making
 the change of variable $\mathbf{y} = \mathbf{X}^t(\x)$ in the expression for~$I_{mod}$:
 \[
  \begin{split}
   I_{mod}[\Q^t] &= \int_\Omega \left(\phi(|\nabla\Q|)+ t\frac{\phi^\prime(|\nabla\Q|)}{|\nabla\Q|}
     \partial_kQ_{ij}\, \partial_p Q_{ij} \, \partial_k X_p + \frac{1}{L}f_B(\Q)\right) \\
     &\qquad\qquad\qquad \cdot\left(1 - t\div\mathbf{X}\right) \d\mathbf{y} + \mathrm{o}(t) 
  \end{split}
 \]
 Now, we differentiate both sides with respect to~$t$ and evaluate for~$t=0$.
 Due to the minimality of~$\Q$, the left-hand side vanish, and hence, by rearranging, we obtain
 \begin{equation} \label{stressenergy}
  \int_\Omega \frac{\phi^\prime(|\nabla\Q|)}{|\nabla\Q|}
     \partial_kQ_{ij}\, \partial_p Q_{ij} \, \partial_k X_p~dV 
     = \int_\Omega e_L(\Q)(\div\mathbf{X})~dV
 \end{equation}
 In particular for a radial vector field $\mathbf{X}(\x)=\gamma(|\x|) \x$
 where $\gamma$ is a smooth, compactly supported scalar function, we compute
 \[
  \partial_k X_p=\gamma^\prime(|\x|)|\x|\nu_p\nu_k+\gamma(|\x|)\delta_{kp}, 
  \quad \div\mathbf{X} =\gamma^\prime(|\x|)|\x|+ d\gamma(|\x|)
 \]
 where~$\boldsymbol{\nu} :=\frac{\x}{|\x|}$. Thus, \eqref{stressenergy} becomes
 \begin{equation*} 
  \begin{split}
    &\int_{\Omega} \gamma^\prime(|\x|)|\x|
    \frac{\phi^\prime(|\nabla\Q|)}{|\nabla\Q|}\abs{\partial_{\boldsymbol{\nu}}\Q}^2 
       + \int_{\Omega}\gamma(|\x|) \phi^\prime(|\nabla\Q|)|\nabla\Q| \\
    &\qquad\qquad\qquad = \int_{\Omega} \gamma^\prime(|\x|)|\x| e_L(\Q)
       + d \int_{\Omega} \gamma(|\x|)e_L(\Q).
  \end{split}
 \end{equation*}
 Letting $\gamma\to\chi_{(0, \, \rho)}$, where~$\rho\in(r, \, R)$, for a.e.~$\rho$ we get
 \[
  \begin{split}
    & - \rho \int_{\partial B_{\rho}} \frac{\phi^\prime(|\nabla\Q|)}{|\nabla\Q|}|
         \partial_{\boldsymbol{\nu}}\Q|^2  +\int_{B_{\rho}} \phi^\prime(|\nabla\Q|)|\nabla\Q| \\
    &\qquad\qquad = - \rho \int_{\partial B_{\rho}}\left(\phi(|\nabla\Q|) +\frac{1}{L}f_B(\Q)\right)
    + d \int_{B_{\rho}}\left(\phi(|\nabla\Q|) +\frac{1}{L}f_B(\Q)\right)
  \end{split}
 \]
 Rearranging the terms, we obtain
 \[
  \begin{split}
   &\rho \int_{\partial B_{\rho}} e_L(\Q)
   + \int_{B_{\rho}}\left(\phi^\prime(|\nabla\Q|)|\nabla\Q| - p \phi(|\nabla\Q|)\right) \\
   &= \rho \int_{\partial B_{\rho}} \frac{\phi^\prime(|\nabla\Q|)}{|\nabla\Q|}
      |\partial_{\boldsymbol{\nu}}\Q|^2  
      + (d- p) \int_{B_{\rho}} e_L(\Q) + \frac{p}{L} \int_{B_{\rho}}f_B(\Q)
  \end{split}
 \]
 Using this equality and the formula for the derivative of the 
 rescaled energy~\eqref{derivative_rescaleden}, we get:
 \[
  \begin{split}
   \frac{\d}{\d\rho} \left\{\rho^{p-d} \int_{B_{\rho}} e_L(\Q)\right\}
   &+ \rho^{p-d-1}\int_{B_{\rho}}\left(\phi^\prime(|\nabla\Q|)|\nabla\Q|- p\phi(|\nabla\Q|)\right) \\
   &=\rho^{p-d} \int_{\partial B_{\rho}} \frac{\phi^\prime(|\nabla\Q|)}{|\nabla\Q|}
      |\partial_{\boldsymbol{\nu}}\Q|^2 + \frac{p\rho^{p-d-1}}{L}\int_{B_{\rho}} f_B(\Q)
  \end{split}
 \]
 Integrating the above formula with respect to~$\rho$, and keeping in mind the 
 definition of~$\xi$ \eqref{monotonicity_error}, we get:
 \begin{equation}\label{mono}
  \begin{split}
   &{R}^{p-d} \int_{B_{R}} e_L(\Q)-{\rho}^{p-d} \int_{B_r} e_L(\Q) +
    \int_r^R \left(\rho^{p-d-1}\int_{B_\rho}\xi(|\nabla\Q|)\right) \d\rho \\
   &\qquad\qquad \geq \int_{B_R\setminus B_r} |\mathbf{y}-\x|^{p-d} 
   \frac{\phi^\prime(|\nabla\Q|)}{|\nabla\Q|}\abs{\frac{\partial\Q}{\partial|\mathbf{y}-\x|}}^2,
  \end{split}
 \end{equation}
 so the first inequality in~\eqref{eq:G1} follows. To conclude the proof we observe that,
 by Assumption~\eqref{hp:subquadratic}, the function~$\xi$ 
 is bounded from above. In particular, we have $\int_{B_\rho}\xi(|\nabla\Q|)\leq C\rho^d$
 for some constant~$C$ that only depends on~$\phi$, $d$, so the second line of~\eqref{eq:G1}
 follows.
\end{proof}

\begin{lemma} \label{lem:3}
 Let $\Q_L$ be a global minimizer of $I_{mod}$ in the admissible space $\mathcal{A}$ defined in Equation~\eqref{eq:3}. 
 Then there exists a sequence $L_k \to 0$ such that $\Q_{L_k}\rightharpoonup\Q_0$ 
 weakly in $W^{1,\phi}\left(\Omega; S_0 \right)$, where
 $\Q_0$ is a $\phi$-minimizing uniaxial tensor-valued map as defined
 above (Definition~\ref{def:harmonic}). Moreover, for any smooth 
 subdomain~$\omega\subset\!\subset\Omega$ there holds
 \[
  \int_{\omega} \phi(|\nabla\Q_{L_k}|)\to \int_\omega\phi(|\nabla\Q_0|),
  \qquad \frac{1}{L_k} \int_{\omega} f_B(\Q_{L_k})\to 0.
 \]
\end{lemma}
\begin{proof}
The proof closely follows Lemma $3$ in \cite{majumdarzarnescu}.
Let $\Q_*$ be a $\phi$-minimizing uniaxial tensor-valued map, 
in the sense of Definition~\ref{def:harmonic}. (The existence of such a map
follows by routine arguments, based on the direct method of the calculus of variations).
We note that $\Q_*\in\mathcal{A}$ and since $\Q_*(\x)\in\NN$ for a.e.~$\x$
(recall $\NN$ is the set of minimizers of the bulk potential $f_B$ in \eqref{eq:2}), 
we have $f_B \left(\Q_*\right) = 0$ a.e. in~$\Omega$. 
We get the following chain of inequalities
\begin{equation}
\label{eq:s1}
\int_{\Omega} \phi(|\nabla \Q_L|)dV \leq 
\int_{\Omega} \left(\phi(|\nabla \mathbf{Q}_L|) + \frac{1}{L} f_B(\Q_L)\right)~dV 
\leq \int_{\Omega} \phi(|\nabla \Q_*|)~dV.
\end{equation}
This shows that the $L^\phi$-norms of $\nabla\Q_L$ are uniformly bounded in the parameter $L$ 
and hence, we can extract a weakly convergent subsequence $\Q_{L_k}$
such that $\Q_{L_k} \rightharpoonup \Q_0$ in $W^{1,\phi}$, for some $\Q_0\in \mathcal{A}$ as $L_k \to 0$.
From Equation \eqref{eq:s1}, we can easily see that $\int_{\Omega} f_B\left(\Q_{L_k} \right) \to 0$ as $L_k \to 0$, so that taking account of the fact that $f_B(\Q)\geq 0$ for all $\Q \in S_0$, we have that $f_B(\Q_{L_k})$ converges to zero pointwise almost everywhere in $\Omega$, up to extraction of a non-relabelled subsequence. The bulk potential $f_B(\Q)=0$ if and only if $\Q \in \NN$ (\cite{ejam2010,ballnotes}).
Hence, the weak limit $\Q_0$ is of the form
\begin{equation}
\label{eq:s3}
\mathbf{Q}_{0}(\x) = s_+\left( \n_{0}(\x)\otimes \n_{0}(\x) - I/3 \right);
\qquad \n_{0}(\x)\in S^2 ~\textrm{for $\x\in\Omega$}.
\end{equation}

Using the convexity of~$\phi$, which implies the weak lower
semicontinuity of the corresponding integral functional, from~\eqref{eq:s1} we get
\begin{equation}
\label{eq:s2bis}
\int_{\Omega} \phi(|\nabla \Q_{0}|)~dV \leq
\liminf_{k\to+\infty} \int_{\Omega} \phi(|\nabla \Q_{L_k}|)~dV \leq
\int_{\Omega} \phi(|\nabla \Q_{*}|)~dV
\end{equation}
and, because~$\Q_*$ is $\phi$-minimizing harmonic,
\begin{equation}
\label{eq:s2}
\int_{\Omega} \phi(|\nabla \Q_{*}|)~dV \leq
\int_{\Omega} \phi(|\nabla \Q_{0})|)~dV.
\end{equation}
These inequalities, together, imply that $\Q_0$ is $\phi$-minimizing harmonic and that
\begin{equation} \label{eq:s_phi}
\int_{\Omega} \phi(|\nabla \Q_{0}|)~dV =
\lim_{k\to+\infty}\int_{\Omega} \phi(|\nabla \Q_{L_k}|)~dV.
\end{equation}
Then, passing to the limit into~\eqref{eq:s1}, we 
deduce that $\frac{1}{L_k}\int_\Omega f_B(\Q_{L_k})\to 0$. Finally, if there 
existed a smooth subdomain~$\omega\subset\!\subset\Omega$ and a further subsequence
$L_{k_j}\to 0$ such that
\[
 \int_\omega \phi(|\nabla \Q_{0}|)~dV < \lim_{j\to+\infty} 
 \int_\omega \phi(|\nabla \Q_{L_{k_j}}|)~dV,
\]
then we would have
\[
 \begin{split}
  \int_{\Omega\setminus\omega} \phi(|\nabla \Q_{0}|)~dV &\stackrel{\eqref{eq:s_phi}}{=}
  \lim_{j\to+\infty} \int_{\Omega} \phi(|\nabla \Q_{L_{k_j}}|)~dV - \int_\omega \phi(|\nabla \Q_{0}|)~dV \\
  &> \lim_{j\to+\infty} \int_{\Omega\setminus\omega} \phi(|\nabla \Q_{L_{k_j}}|)~dV,
 \end{split}
\]
which contradicts the lower semi-continuity. Therefore, we must have 
\[
 \int_\omega \phi(|\nabla \Q_{0}|)~dV \geq \limsup_{k\to+\infty} 
 \int_\omega \phi(|\nabla \Q_{L_{k}}|)~dV
\]
for any smooth subdomain~$\omega\subset\!\subset\Omega$, whence the lemma follows.
\end{proof}

\subsection{Splitting type functionals}
\label{sect:splitting}

The aim of this section is to prove the following result:

\begin{proposition} \label{prop:C1,a}
 Let~$L>0$ be fixed, and let~$\bar{\Q}_L$ be a minimizer of the functional~$I_{LdG}$
 defined by~\eqref{eq:2}. Suppose that the assumptions~\eqref{hp:first}--\eqref{hp:last}
 are satisfied. Then, there exists~$\alpha\in (0, \, 1)$ (only depending on the characteristics of~$\phi$)
 such that $\bar{\Q}_L\in C^{1,\alpha}_{\mathrm{loc}}(\Omega)$.
\end{proposition}

Throughout this section, $L$ is fixed, so we will write~$\bar{\Q}$ instead of~$\bar{\Q}_L$.
We deduce the proposition from the regularity results in~\cite{diestrover} 
(see Theorem~\ref{generalgrowth} above). However, these results apply 
to a problem without right-hand side (i.e., the case~$f_B = 0$).
In order to reduce to this case, we consider a ball $B(\x_0, R)\subseteq\Omega$.
We compare $\bar{\Q}$ with the solution  $\P$ of the $\phi$-harmonic system with boundary data $\bar{\Q}$ on the boundary of the ball $B(\x_0, R)$:
\begin{equation} \label{frozen}
 \div\Big( \frac{\phi'(|\nabla \P|)}{|\nabla \P|} \nabla \P\Big)=0
\end{equation}
in the class $\bar{\Q}+W^{1,\phi}_0 (B(\x_0, R), S_0)$.
The existence and uniqueness of~$\P$ follows by the strict convexity of~$\phi$,
by a routine application of the direct method of the calculus of variations.
Recall that, by the maximum principle (Lemma~\ref{lemma:max_principle}),
$\bar{\Q}$ is bounded independently of~$L$.

\begin{proposition} \label{lemma:frozen_excess}
 Let $\P$ be a solution of the system~\eqref{frozen}
 in the space $\bar{\Q}+W^{1,\phi}_0 (B(\x_0, R), S_0)$.
 For every $\epsilon>0$ there exists $C_{\epsilon}>0$ such that:
 \begin{equation*}
 \mI{B(\x_0, R)}|\V(\nabla\bar{\Q})-\V(\nabla \P)|^2 \d x
 \le \frac{1}{L} \left[\epsilon \mI{B(\x_0, R)} \phi(|\nabla \bar{\Q}|)+C_{\epsilon}\phi^*(R)\right] \!,
 \end{equation*}
 where~$\V$ is defined by~\eqref{eq:defV} and~$\phi^*$  is defined by \eqref{phi*}.
\end{proposition}
\begin{proof}
 Let us set $\A(\nabla\Q):=\frac{\phi^\prime(|\nabla \Q|)}{|\nabla\Q|}$. 
 We observe that $\bar{\Q}$ is a solution of the Euler-Lagrange system:
 \begin{equation*}
  -\div \Big( \frac{\phi'(|\nabla \bar{\Q}|)}{|\nabla \bar{\Q}|} \nabla \bar{\Q}\Big)
  = -\frac1L \nabla_{\Q}f_B(\bar{\Q}) + \frac{B}{3L}|\bar{\Q}|^2 \mathbf{I}
 \end{equation*}
 We consider the difference between the above system and~\eqref{frozen},
 testing with $\bar{\Q}-\P$.
 Using \eqref{eq:equivalence}, we can rewrite the right-hand side as follows:
 \begin{equation*}
  \mI{B_{R}}(\A( \nabla \bar{\Q})-\A(\nabla \P)) : (\nabla \bar{\Q}-\nabla \P) \d x\sim \mI{B_{R}} |\V(\nabla \bar{\Q})-\V(\nabla\P)|^2 \d x
 \end{equation*}
 The right-hand side can be estimated taking into account that 
 $\tr(\bar{\Q} - \P) = 0$, $f_B$ is 
 a polynomial in $\Q$ and $\bar{\Q}$ is bounded by the maximum principle 
 (Lemma~\ref{lemma:max_principle}); this yields
 \begin{equation} \label{frozen_excess}
  \begin{split}
   &\mI{B_{R}}|\V(\nabla \bar{\Q})-\V(\nabla\P)|^2 \d x \\
   &\qquad\qquad \lesssim
   \frac{1}{L} \abs{\mI{B_{R}} \nabla_{\Q}f_B(\bar{\Q}) : (\bar{\Q}-\P) \, \d x}
   \leq \frac{C}{L} \mI{B_{R}} |\bar{\Q} -  \P | \d x .
  \end{split}
 \end{equation}
 On the other hand, using Young \eqref{Young} and Poincar\'e inequalities 
 and \eqref{subadd}, we obtain
  \begin{align*}
   \frac{1}{L} \mI{B_{R}} |\bar{\Q} -  \P | \d x 
   &= \frac{C}{L} \mI{B_{R}} R \frac{|\bar{\Q} -  \P|}{R} \, \d x
   \le \frac{1}{L} \left(\epsilon \mI{B_R} 
   \phi\left(\frac{|\bar{\Q} - \P|}{R}\right)  + C_{\epsilon}\phi^*(R)\right) \\
   &\lesssim \frac{1}{L} \left(\epsilon \mI{B_R} 
   \phi(|\nabla\bar{\Q} - \nabla\P|)  + C_{\epsilon}\phi^*(R)\right) \\
   &\lesssim \frac{1}{L} \left(\epsilon \mI{B_{R}}
   \left(\phi(|\nabla \bar{\Q}|) + 
   \phi(|\nabla\P|) \right) + C_{\epsilon}\phi^*(R)\right)
 \end{align*}
 The minimality of $\P$ implies that
 the average of $\phi(|\nabla\P|)$ is bounded from above 
 by the average of $\phi(|\nabla \bar{\Q}|)$, and hence, the proposition follows.
\end{proof}

\begin{proof}[of Proposition~\ref{prop:C1,a}]
 We first show that~$\bar{\Q}\in C^\alpha_{\mathrm{loc}}(\Omega)$.
 Let $B(\x_0, R)\subseteq\Omega$ be a fixed ball, and let~$\P$
 be the solution of~\eqref{frozen} on~$B(\x_0, R)$ 
 such that $\P = \bar{\Q}$ on $\partial B(\x_0, R)$.
 Let~$B_{\rho}$ be a ball of radius~$\rho$ with~$B_\rho\subseteq B(\x_0, R)$.
 We apply the previous estimate (Proposition~\ref{lemma:frozen_excess}), so to deduce
 \begin{equation*}
  \begin{split}
   \mI{B_\rho} \phi(|\nabla \bar{\Q}|) \d x 
   & \lesssim \mI{B_{\rho}} |\V(\nabla\bar{\Q})-\V(\nabla\P)|^2 \d x 
      + \mI{B_{\rho}} |\V(\nabla\P)|^2 \d x \\
   & \lesssim  L^{-1} \left(\frac{R}{\rho}\right)^{d} \left[\epsilon \mI{B_{R}} 
      \phi(|\nabla \bar{\Q}|)+C_{\epsilon}\phi^*(R)\right]
      + \sup_{B_\rho} \phi(|\nabla\P|)
  \end{split}
 \end{equation*}
 Since $\P$ is $\phi$-harmonic, we can use the $L^{\infty}-L^1$ estimate proven in~\cite{diestrover} 
 (Theorem~\ref{generalgrowth}, Eq. \eqref{excess}):
 \begin{equation*}
  \begin{split}
   \mI{B_\rho} \phi(|\nabla \bar{\Q}|) \d x 
   \lesssim L^{-1} \left(\frac{R}{\rho}\right)^{d} \left[\epsilon \mI{B_{R}} 
      \phi(|\nabla \bar{\Q}|)+C_{\epsilon}\phi^*(R)\right]
      + \mI{B_{R}} \phi(|\nabla\P|) \d x \\
  \end{split}
 \end{equation*}
 The last term can be bounded from above by 
 $\mI{B_R} \phi(|\nabla\bar{\Q}|)\d x$ using the minimality of $\P$.
 The term $\phi^*(R)$ is also bounded from above, because~$R\leq\mathrm{diameter}(\Omega)<+\infty$.
 Let us consider the quantity $h(r) := \int_{B_r} \phi(|\nabla\bar{\Q}|)\d x$.
 We have proven that for any  $0 < \rho < R/2$  it holds:
 \[
  h(\rho) \leq C_1 \left[\left(\frac{\rho}{R}\right)^{d} + \epsilon L^{-1}\right] h(R) 
  + C_\epsilon {R}^{d} 
 \]
 for some positive constants~$C_1$ that only depends on~$\phi$, $d$.
 By modifying the value of~$C_1$, if necessary, we can make sure that the same
 inequality holds for any~$0 < \rho < R$. 
 We apply Giaquinta's Lemma \cite[Lemma~2.1, p.~86]{Giaquinta} and conclude that,
 if we choose~$\epsilon$ small enough, there holds 
 \[
  h(\rho) \leq c_\sigma \left[\left(\frac{\rho}{R}\right)^{\sigma} h(R) 
  + C_\epsilon R^{\sigma} \right]
 \]
 for all $0 < \rho < R$ and all~$0<\sigma<d$,
 where~$c_\sigma> 0$ depends only on~$\phi$, $d$, $\sigma$.
 Thanks to Morrey's characterisation of H\"older continuous functions, we conclude that
 $\bar{\Q}\in C^\alpha(B(\x_0, R/2))$ for any~$\alpha\in (0, \, 1]$.
 
 One we know that~$\bar{\Q}$ is locally H\"older continuous, we can prove
 that $\bar{\Q}\in C^{1,\alpha}_{\mathrm{loc}}(\Omega)$ by adapting the arguments 
 in \cite[Proposition~5.1, third step]{diestroverJDE}.
 (In the power case the proof was done in \cite[Lemma 5]{DuzaarMingione}.)
 Let $B_\rho\subseteq B(\x_0, R)$ be balls contained in~$\Omega$.
 Let~$\P$ be the solution of the system~\eqref{frozen} on the smaller ball~$B_\rho$, 
 in the class $\bar{\Q} +W^{1,\phi}(B_\rho, \, S_0)$.
 In the proof of Proposition~\ref{lemma:frozen_excess}
 (see Equation~\eqref{frozen_excess}), we have shown that 
 \begin{equation} \label{frozen-f}
  \mI{B_\rho}|\V(\nabla\bar{\Q})-\V(\nabla \P)|^2 \d x
  \le \frac{C}{L} \mI{B_\rho} | \bar{\Q}-\P| \d x .
 \end{equation}
 Since $\bar{\Q}$ is locally $\alpha$-H\"older continuous, we have:
 \begin{equation*}
   \sup_{y,z\in B_{\rho}}|\bar{\Q}(y)-\bar{\Q}(z)|\le C_{\alpha, L} \rho^\alpha
 \end{equation*}
 for an arbitrary~$\alpha\in(0, \, 1)$ and some constant~$C_{\alpha, L}$ 
 depending on~$\alpha$, $L$ but not on~$\rho$.
 We can apply the convex-hull property for $\P$: the image of $\P(B_{\rho})$ is 
 contained in the convex hull of $\P(\partial B_{\rho})=\bar{\Q}(\partial B_{\rho})$.
 Therefore, keeping in mind that $\P = \bar{\Q}$ on~$\partial B_\rho$,
 \begin{equation} \label{convexhull}
  \|\P - \bar{\Q}\|_{L^\infty(B_\rho)}\le
  C_{\alpha, L} \rho^\alpha.
 \end{equation}

 Let us consider the excess functional $\Phi$, defined by
 \begin{align*}
    \Phi(\bar{\Q}, B_\rho) := \mI{B_\rho} |\V(\nabla\bar{\Q})
       - (\V(\nabla\bar{\Q}))_{\x_0, \rho}|^2 \d x.
 \end{align*}
 Then, for any~$\kappa\in (0, 1/2)$, we have 
 \begin{equation*}
  \begin{split}
   \Phi(\bar{\Q}, B_{\kappa \rho}) &\lesssim
   \mI{B_{\kappa \rho}}|\V(\nabla\bar{\Q})-\V(\nabla \P)|^2 \d x  + \Phi(\P, B_{\kappa \rho}) \\
   &\qquad\qquad + |(\V(\nabla\bar{\Q}))_{\x_0, \rho} - (\V(\nabla \P))_{\x_0, \rho}|^2 \\
   &\leq 2\mI{B_{\kappa \rho}}|\V(\nabla\bar{\Q})-\V(\nabla \P)|^2 \d x + \Phi(\P, B_{\kappa \rho}) \\
   &\leq 2\kappa^{-d}\mI{B_{\rho}}|\V(\nabla\bar{\Q})-\V(\nabla \P)|^2 \d x + \Phi(\P, B_{\kappa \rho}) \\
   &\stackrel{\eqref{frozen-f}-\eqref{convexhull}}{\leq} 
    2C_{\alpha, L}\kappa^{-d} \rho^\alpha + \Phi(\P, B_{\kappa \rho})
  \end{split}
 \end{equation*}
 It follows from Theorem \ref{generalgrowth} that there exists 
 $\gamma>0$ and $c>0$ only depending on $d$ and the characteristics of $\phi$,
 such that $\Phi(\P, B_{\kappa \rho}) \leq c\, \kappa^{2\gamma} \Phi(\P, B_\rho)$. Thus
  \begin{align*}
    \Phi(\bar{\Q}, B_{\kappa \rho}) 
    \leq 2C_{\alpha, L}\kappa^{-d} \rho^\alpha + c\, \kappa^{2\gamma}  \Phi(\P, B_\rho)
    \leq 2C_{\alpha, L}\kappa^{-d} \rho^\alpha + c\, \kappa^{2\gamma}  \Phi(\bar{\Q}, B_\rho)
  \end{align*}
 (the excess of~$\P$ can be controlled by the excess of~$\bar{\Q}$ using the minimality of~$\P$).
 Now, choose $\kappa\in (0, \, 1/2)$ such that $c\kappa^{2\gamma} \leq 1/2$. Then
 \begin{align*}
    \Phi(\bar{\Q}, B_{\kappa \rho}) \leq \frac 12 \Phi(\bar{\Q}, B_{\rho})
    + c_{\kappa, \alpha, L} \rho^\alpha.
 \end{align*}
 By iterating the previous inequality, for each~$j\in\N$ we obtain
 \begin{align*}
   \Phi(\bar{\Q}, B_{\kappa^j \rho}) \leq \frac 1{2^j} \Phi(\bar{\Q}, B_{\rho})
   + \sum_{i=0}^{j-1} 2^{-j}c_{\kappa, \alpha, L} \rho^\alpha
    \leq \frac 1{2^j} \Phi(\bar{\Q}, B_{\rho})
   + 2c_{\kappa, \alpha, L} \rho^\alpha.
 \end{align*}
 Thus, there exists $\beta>0$, $c_\beta>0$ (depending on~$\kappa$)
 such that for all $r \in (0, \, \rho)$ there holds
 \begin{align}\label{eq:excess-decay2}
  \Phi(\bar{\Q}, B_r) &\leq c_\beta \left(\frac{r}{\rho}\right)^\beta 
  \Phi(\bar{\Q}, B_{\rho}) + 2c_{\kappa, \alpha, L} \rho^\alpha.
 \end{align}
 Recall that $B_\rho$ is an arbitrary ball with $B_\rho\subseteq B(\x_0, R)$.
 
 Let~$B(\x, r) \subseteq B(\x_0, R/2)$ be an arbitrary ball of radius 
 $r < R/4$. We choose $\rho := R/2$ and $s := (r\rho)^{1/2} = (rR/2)^{1/2}$,
 so that $B(\x, r) \subseteq B(\x, s) \subseteq B(\x, \rho) \subseteq B(\x_0, R)$.
 We apply~\eqref{eq:excess-decay2} first on the balls~$B(\x, r) \subseteq B(\x, s)$, 
 then on~$B(\x, s) \subseteq B(\x, \rho)$. We obtain
 \begin{align*}
   \Phi(\bar{\Q}, B(\x, r)) &\leq c_\beta \left(\frac{2r}{R}\right)^{\beta/2} 
    \Phi(\bar{\Q}, B(\x, s)) + 2c_{\kappa, \alpha, L} \left(\frac{rR}{2}\right)^{\alpha/2} \\
    &\leq c^2_\beta \left(\frac{r}{R}\right)^{\beta} \Phi(\bar{\Q}, B(\x, R/2)) \\
    &\qquad\qquad + 2c_{\kappa, \alpha, L} c_\beta \left(\frac{2r}{R}\right)^{\beta/2} 
    \left(\frac{R}{2}\right)^{\alpha/2}
    + 2c_{\kappa, \alpha, L} \left(\frac{rR}{2}\right)^{\alpha/2}.
 \end{align*}
 Since $\Phi(\bar{\Q}, B(\x, R/2))\lesssim \Phi(\bar{\Q}, B(\x_0, R))$
 and~$\alpha\in (0, 1)$ is arbitrary, for fixed $\x_0$, $R$, $L$ the right-hand
 side grows as~$r^{\beta/2}$, when~$r$ is small. This estimate, combined with Campanato's
 characterisation of H\"older functions, proves that
 $\V(\nabla\bar{\Q})\in C^{\beta/4}(B(\x_0, R/4))$. 
 Since the function~$\V$ is invertible and the inverse~$\V^{-1}$ 
 is $\mu$-H{\"o}lder continuous, for some $\mu>0$ which only depends on the characteristics
 of~$\phi$ \cite[Lemma~2.10]{diestrover}, we conclude that 
 $\nabla\bar{\Q}\in C^{\mu\beta/2}_{\mathrm{loc}}(\Omega)$.
\end{proof}

\subsection{Subharmonicity}

The aim of this subsection is to prove that, given a minimizer~$\bar{\Q}_L$
of the functional~\eqref{eq:2}, $\phi(|\nabla\bar{\Q}_L|)$ is a subsolution of 
a \emph{scalar} elliptic problem. This was the so-called Bernstein-Uhlenbeck trick used in the power 
case \cite{uhlenbeck} for the case $p\ge 2$, adapted for the general growth in~\cite{diestrover}.

\begin{theorem}
  \label{thm:subsol}
  Let $\phi$ satisfy the assumptions~\eqref{hp:first}--\eqref{hp:last}.
  Let $\bar{\Q}_L$ be a minimizer of the functional~\eqref{eq:2} and let $B = B(\x_0, R)$ be a 
  ball with $2B := B(\x_0, 2R)\subseteq\Omega$.
  Then there exists $\mathbf{G}\colon\R^d\otimes S_0\to\mathbb{R}^{d \times d}$
  that is uniformly elliptic and satisfies
  \begin{align*} 
    \int_{2B} G^{kl}(\nabla\bar{\Q}_L) \partial_l \left(\phi(\abs{\nabla\bar{\Q}_L})\right)
    \partial_k \eta\, \d x &\leq - \frac{1}{L} \int_{2B} 
    \frac{\partial^2 f_B(\bar{\Q}_L)}{\partial Q_{ij}\partial Q_{hk}} 
    \partial_\ell\bar{Q}_{ij} \partial_\ell\bar{Q}_{hk} \, \eta \, \d x
  \end{align*}
  for all $\eta \in C^1_0(2B)$ such that~$\eta\ge 0$. Moreover, for all
  $\mathbf{D}\in \mathbb{R}^d\otimes S_0$ and all $\xi\in\mathbb{R}^n$ there holds
  \begin{align*}
    \alpha_0 \abs{\xi}^2 \leq \sum_{k,l}
    G^{kl}_\lambda(\mathbf{D}) \xi_k \xi_l \leq \alpha_1\abs{\xi}^2,
  \end{align*}
  where $\alpha_0$, $\alpha_1$ are positive constants that only depend 
  on the characteristics of~$\phi^\prime$.
\end{theorem}

Before proving the theorem, we need to prove the existence of second derivatives of~$\bar{\Q}_L$ç
because in the following computations we will encounter terms of the form
$\int \eta\, \abs{\nabla\V(\nabla\Q)}^2\,\d x$. We
 can apply the results in~\cite{DieningEttwein}
to deduce higher integrability and existence of second derivatives; 
in particular, we have $\V(\nabla\bar{\Q}_L)\in W^{1,2}_{\mathrm{loc}}(\Omega)$.

To prove Theorem~\ref{thm:subsol}, it is convenient to work on an approximated system.
For $\lambda>0$ and $t\geq 0$ we define
\begin{equation}
\phi'_{\lambda}(t)= \frac{\phi(\lambda+t)}{\lambda+t}t
\end{equation}
and
\begin{align}
  \label{eq:def_omega_l}
  \omega_\lambda(t) &:= \frac{\phi_\lambda''(t)\,t -
    \phi_\lambda'(t)}{\phi_\lambda'(t)}.
\end{align}
It follows from assumption~\eqref{hp:Delta2prime} that there exist positive
constants $c_0$, $c_1$ (the characteristics of~$\phi^\prime$) such that
\begin{align}
  \label{eq:omega_l}
  c_0 - 1 \leq \omega_\lambda(t)\leq c_1 - 1
\end{align}
for all $t\geq 0$ and all $\lambda>0$.
Given $L>0$ and a critical point~$\bar{\Q}$ of the functional~\eqref{eq:2},
we consider the approximated system in~$\Q_\lambda$:
\begin{equation} \label{approx_EL}
-\div\Big(\frac{\phi'_{\lambda}(|\nabla\Q_\lambda|)}
 {|\nabla\Q_\lambda|}\nabla\Q_{\lambda, ij}\Big) = 
 -\frac{1}{L}\left( \frac{\partial f_B}{\partial Q_{ij}}(\bar\Q)
+ \frac{B}{3}|\bar\Q|^2 \delta_{ij}\right) 
\end{equation}
subject to the boundary conditions~$\Q_\lambda = \bar{\Q} = \Q_b$ on~$\partial\Omega$. 
(Note that the right-hand side is a function of~$\bar\Q$, not of~$\Q_\lambda$, 
and hence can be trated as a given source term.)
Since $\phi_\lambda$ is strictly convex, this system has a unique
solution~$\Q_\lambda$ for any given~$L$, $\bar\Q$.
Moreover, $\Q_\lambda$ converges weakly to~$\bar\Q$ in~$W^{1,\phi}$ as~$\lambda\to 0$ 
(see e.g.~\cite[Theorem~4.6]{diestrover}). The next results shows that
$\phi_\lambda(\abs{\nabla \Q_\lambda})$
is a subsolution to a uniformly elliptic problem, where the constants
of ellipticity do no depend on~$\lambda>0$; we can then recover Theorem~\ref{thm:subsol}
by passing to the limit~$\lambda\to 0$.

\begin{proposition}
  \label{prop:subsol}
  Let $\phi$ satisfy assumptions \eqref{hp:first}--\eqref{hp:last}. Let $\Q_\lambda$ be a 
  solution of the approximated system~\eqref{approx_EL} and let $B$ be a ball with $2B \subset \Omega$.
  Then there exists $\mathbf{G}_\lambda\colon\R^d\otimes S_0\to\mathbb{R}^{d \times d}$
  which is uniformly elliptic such that
  \begin{equation*}
    \label{eq:subsol}
    \begin{split}
    &\int_{2B} \sum_{kl} \Bigg[
    G^{kl}_\lambda(\nabla \Q_\lambda)
    \partial_l \big( \phi_\lambda(\abs{\nabla \Q_\lambda}) \big)
    \Bigg] \partial_k \eta\,\d x \\
    &\quad \leq -c\! \int_{2B} \eta\, \abs{\nabla\V_\lambda(\nabla \Q_\lambda)}^2\,\d x
      -\frac{1}{L} \int_{2B} \eta \, 
      \frac{\partial^2 f_B(\bar{\Q})}{\partial Q_{ij}\partial Q_{hk}}
      \partial_\ell\bar{Q}_{ij} \partial_\ell\bar{Q}_{hk} \, \d x
    \end{split}
  \end{equation*}
  holds for all $\eta \in C^1_0(2B),\eta\ge 0$. Moreover, there holds
  \begin{align*}
    \min \{c_0,1\} \abs{\xi}^2 \leq \sum_{k,l}
    G^{kl}_\lambda(\mathbf{D}) \xi_k \xi_l \leq (c_1+1) \abs{\xi}^2
  \end{align*}
  for all $\mathbf{D}\in \mathbb{R}^{d}\otimes S_0$ and all $\xi\in \mathbb{R}^d$,
  where $c_0$, $c_1>0$ are the constants from~\eqref{eq:omega_l}.
\end{proposition}

\begin{proof} The proof parallels the one presented in ~\cite{diestrover}(Lemma 5.4) , with an additional lower order term.
  Let $\eta \in C^1_0(2B)$. Let $B_R$ be a ball of radius~$R$ and let
  $h \in \mathbb{R}^d \setminus \{0\}$ with $\abs{h} \leq \min\{
  \dist(\spt(\eta), \partial(2B)), 1\}$. Let~$\tau_h$ be the finite difference operator,
  defined by
  \[
   (\tau_h\mathbf{F}) (\x) := \mathbf{F}(\x + h) - \mathbf{F}(\x) 
   \qquad \textrm{for } \x\in\R^d
  \]
  for an arbitrary function~$\mathbf{F}\colon\R^d\to S_0$.
  Define $\xi :=\abs{h}^{-2} \tau_{-h}(\eta \tau_h \Q_\lambda)$, then $\xi\in
  W^{1,\phi}_0(2B, S_0)$, so $\xi$ is an admissible test function.
  By multiplying both sides of the system by~$\xi$, and using that~$\tr\xi = 0$, we obtain
  \begin{align*}
    -\frac{\abs{h}^{-2}}{L} \int &\tau_h ((\nabla_{\Q} f_B)(\bar{\Q}))\cdot\tau_h(\bar{\Q}) \eta \, dx  \\
    & = \int \abs{h}^{-2} \sum_{j,k} \tau_h
    \big(A_\lambda^{jk}(\nabla \Q_\lambda) \big) \partial_k (\eta\,
    \tau_h \Q_{\lambda,j})\,dx
    \\
    &= \int \abs{h}^{-2} \sum_{j,k} \tau_h \big(A_\lambda^{jk}(\nabla
   \Q_\lambda) \big) (\partial_k \eta) \tau_h \Q_{\lambda,j}\,dx
    \\
    &\quad + \int \abs{h}^{-2} \sum_{j,k}\tau_h
    \big(A_\lambda^{jk}(\nabla \Q_\lambda) \big) \eta\, \tau_h
    \partial_k \Q_{\lambda,j}\,dx =: (I) + (II).
  \end{align*}
  We choose $h := r e_l$ with $l\in\{1, \dots, d\}$ and 
  $0 < r\leq\dist(\spt(\eta),\partial(2B))$. Then, as $r\to 0$, the left hand side converges to
  \[
   -\frac{1}{L}\int  \eta \, \partial_j ((\nabla_{\Q}f_B)(\bar{\Q}))\cdot\partial_j \bar{\Q} \,\d x
   = -\frac{1}{L}\int \eta \, \frac{\partial^2 f_B(\bar{\Q})}{\partial Q_{i\ell}\partial Q_{hk}}
   \partial_j \bar{Q}_{i\ell} \partial_j\bar{Q}_{hk} \,\d x
  \]
  We now have to deal with the terms in the right-hand side, 
  following the strategy in~\cite{diestrover}. We only list the principal steps:
  \begin{align}
    \label{eq:22}
    (II) &\geq c\,\int \eta\, \abs{h}^{-2} \abs{\tau_h
      V_\lambda(\nabla \Q_\lambda)}^2\,dx =: (III).
  \end{align}
  
  \begin{align}
    \label{eq:23}
    (III) &\to \int \eta \abs{\partial_l V_\lambda(\nabla
      \Q_\lambda)}^2\,dx.
  \end{align}
 and  for $r \to 0$
  \begin{align}
    \label{eq:11}
    (I) \to \int \sum_{j,k} \partial_l \big(A_\lambda^{jk}(\nabla
   \Q_\lambda) \big) (\partial_k \eta) \partial_l \Q_{\lambda,j}\,dx
  \end{align}
  and the integral is well defined in~$L^1$.

  After summation over $l=1, \dots, n$
  \begin{equation} \label{eq:24}
   \begin{split}
    &\int \sum_{l,j,k} \partial_l \big(A_\lambda^{jk}(\nabla
    \Q_\lambda) \big) (\partial_k \eta) \partial_l \Q_{\lambda,j}\,dx
    + c\, \int \eta\, \abs{\nabla \big(V_\lambda(\nabla
      \Q_\lambda)\big)}^2\,dx \\
    &\qquad\qquad\qquad \leq 
    -\frac{1}{L}\int \eta \, \frac{\partial^2 f_B(\bar{\Q})}{\partial Q_{i\ell}\partial Q_{hk}}
   \partial_j \bar{Q}_{i\ell} \partial_j\bar{Q}_{hk} \,\d x.
   \end{split}
  \end{equation}
  Note that the constant does not depend on~$\eta \in C^1_0 (2B)$.

 Define $G_\lambda \,:\, \mathbb{R}^{N \times
    n} \to \mathbb{R}^{n \times n}$ by
  \begin{align*}
    G^{kl}_\lambda(\Q) := \delta_{k,l} + \frac{\sum_j (\Q_{jk}\,
      \Q_{jl})}{\abs{\Q}^2}\, \omega_\lambda(\abs{\Q}).
  \end{align*}
  Then, for any index~$k$ we have
  \begin{align*}
    \sum_{jl} \Big( \partial_l \big(A_\lambda^{jk}(\nabla
    \Q_\lambda) \big)
    \partial_l u_{\lambda,j} \Big) &=  \sum_l
    G^{kl}_\lambda(\nabla \Q_\lambda)
    \partial_l \big( \phi_\lambda(\abs{\nabla \Q_\lambda}) \big).
  \end{align*}
  This together with~\eqref{eq:24} implies
  \begin{align*}
    \int \sum_{kl} \Bigg[
    G^{kl}_\lambda(\nabla \Q_\lambda)
    \partial_l \big( \phi_\lambda(\abs{\nabla \Q_\lambda}) \big)
    \Bigg] \partial_k \eta\,dx &\leq -c\, \int \eta\, \abs{\nabla
      V_\lambda(\nabla \Q_\lambda)}^2\,dx \leq 0.
  \end{align*}
    For all $\Q \in \mathbb{R}^{N
    \times n}$ and all $\xi \in \mathbb{R}^n$ holds

  \begin{align*}
    \sum_{k,l} G^{kl}_\lambda(\Q) \xi_k \xi_l &= \abs{\xi}^2 +
    \frac{\abs{\Q \xi}^2}{\abs{\Q}^2}\,\omega_\lambda(\abs{\Q}).
  \end{align*}
  This implies
  \begin{align*}
    \sum_{k,l} G_\lambda^{kl}(\Q) \xi_k \xi_l &\leq \abs{\xi}^2 +
    c_1\, \abs{\xi}^2 = (c_1+1)\, \abs{\xi}^2,
    \\
     \sum_{k,l} G_\lambda^{kl}(\Q) \xi_k \xi_l &\geq \abs{\xi}^2
     \big(1 + \min \{0,c_0-1\}\big) = \min \{c_0,1\},
  \end{align*}
  where $c_0$ and $c_1$ are the constants from~\eqref{eq:omega_l}.
\end{proof}

\subsection{An $L^\infty-L^1$ estimate}

The aim is to prove that global minimizers of $I_{mod}$, 
where $I_{mod}$ is defined in (\ref{eq:2}), converge uniformly to $\Q_0$ everywhere 
away from the singularities of $\Q_0$. This is parallel to the work in \cite{majumdarzarnescu} 
where the authors prove that global minimizers of a Landau-de Gennes energy with
$\phi\left(|\nabla \Q|\right) = \frac{1}{2}|\nabla \Q |^2$ converge uniformly to a 
limiting minimizing harmonic map in the limit $L \to 0$, everywhere away from the
singularities of the minimizing harmonic map and the limit $L\to 0$ in 
\cite{majumdarzarnescu} is equivalent to the asymptotic limit in this paper, modulo some scaling.

The key step in this section is the following result, 
which in inspired by \cite[Lemma~2.3]{changyouwang}.  
\begin{proposition} \label{prop:uniform1}
 There exist positive numbers $r_0$, $\varepsilon$, $\Lambda$ with the following property.
 Let $B(\x^*, \, r)$ be a ball of radius~$r\leq r_0$, let $L > 0$, and 
 let~$\Q_{L}\in W^{1,p}(B(\x^*, \, r), S_0)$ be a minimizer of the 
 functional~$I_{mod}$ on~$B(\x^*, \, r)$. If there holds
 \begin{equation} \label{eq:u2}
  r^{p-d} \int_{B(\x^*, r)} e_{L}(\Q_{L}) dV \leq \varepsilon
 \end{equation}
 then we have
 \begin{equation} \label{eq:u3}
  r^p \sup_{B(\x^*, r/2)}e_{L}(\Q_{L}) \leq \Lambda.
 \end{equation}
\end{proposition}

Before giving the proof of the proposition, we present some auxiliary material.
Let~$\delta_0>0$ be a small parameter, to be specified later. For $\Q\in S_0$ 
such that $\dist(\Q, \, \NN)\leq\delta_0$, we can write in a unique way
\begin{equation*}
\Q = \lambda_1 \n \otimes \n + \lambda_2 \m \otimes \m + \lambda_3 \p \otimes \p 
\end{equation*}
where we assume that $\lambda_1 \leq \lambda_2 <\lambda_3$, $\sum_{i=1}^{3}\lambda_i = 0$
and~$\n$, $\m$, $\p$ are unit vectors. As in \cite[Lemma~6]{majumdarzarnescu}, we define
the projection of~$\Q$ on~$\NN$ as
\begin{equation} \label{eq:p8}
 \Pi\left(\Q \right) = -\frac{s_+}{3}\n \otimes \n - \frac{s_+}{3}\m\otimes \m + \frac{2s_+}{3}\p\otimes\p
\end{equation}
Note that $\Pi(\Q)$ can be written as $\Pi(\Q) = s_+(\u\otimes\u - \mathbf{I}/3)$ 
for some unit vector~$\u$, so $\Pi(\Q)\in\NN$ indeed. In fact, it can be shown that $\Pi(\Q)$
is the nearest-point projection of~$\Q$ onto~$\NN$, that is $|\Q - \P|\geq|\Q - \Pi(\Q)|$
for any~$\P\in\NN$ (see e.g. \cite[Lemma~12]{Canevari3D}). Then, the matrix defined by
\begin{equation}
\label{eq:p9}
\nu(\Q) := \frac{\Q - \Pi\left(\Q \right)}{\left|\Q - \Pi\left(\Q \right) \right|}
\end{equation}
is normal to the manifold~$\NN$ at the point~$\Pi(\Q)$.
We note that
\[
 \left|\Q - \Pi(\Q) \right|^2 = \left(\lambda_1 + \frac{s_+}{3}\right)^2 
 + \left(\lambda_2 + \frac{s_+}{3}\right)^2 
 + \left(\lambda_1 + \lambda_2 + \frac{2s_+}{3}\right)^2 \leq \delta_0^2.
\]

\begin{lemma}
\label{lem:changyouwang1}
 If~$\delta_0>0$ is small enough, then any matrix~$\Q\in S_0$ such that
 $\dist\left(\Q, \NN \right) < \delta_0$ has the following properties:
 \begin{equation} \label{eq:p3}
  \begin{split}
   \alpha_1 |\Q - \Pi(\Q)| &\leq \alpha_2 \sqrt{f_B(\Q)} \leq 
   \left(\frac{\partial f_B}{\partial Q_{ij}}(\Q) + \frac{B}{3}|\Q|^2\delta_{ij} \right)\nu_{ij}(\Q) \\
   &\leq \alpha_3 \sqrt{f_B(\Q)} \leq \alpha_4 |\Q - \Pi\left(\Q \right) |
  \end{split} 
 \end{equation}
 for positive constants $\alpha_1$, $\alpha_2$, $\alpha_3$,
 $\alpha_4$ that only depend on~$A$, $B$, $C$.
\end{lemma}
\begin{proof}
 The lemma follows by \cite[Remark~2.5 and Lemma~3.6]{Canevari2D}; 
 see also \cite[Lemma~6]{majumdarzarnescu}. The crucial point is to show that 
 \[
  \frac{\d^2}{\d t^2}_{|t=0} f_B(\Pi(\Q) + t\nu(\Q)) \geq \alpha 
 \]
 for some~$\alpha>0$ that only depends on~$A$, $B$, $C$;
 then, the lemma follows by Taylor-expanding the function $t\in\R\mapsto f_B(\Pi(\Q) + t\nu(\Q))$
 about~$t=0$.
\end{proof}

\begin{lemma} \label{lem:changyouwang1.5}
 Let~$\delta_0>0$ be as in Lemma~\ref{lem:changyouwang1}. Let~$B_R$ be a ball of radius~$R$.
 Let~$L> 0$ be fixed, and let~$\Q_L\in W^{1,\phi}(B_R, S_0)$ be a solution of 
 the Euler-Lagrange system in~$B_R$ such that
 \begin{equation} \label{hp:closedness}
  \dist(\Q_L(\x), \, \NN)\leq\delta_0 \qquad \textrm{for any } 
  \x\in B_R.
 \end{equation}
 Then, there holds
 \begin{equation} \label{eq:p4}
  \begin{split}
   \frac{1}{L} \int_{B_R} \abs{\Q_L - \Pi(\Q_L)} ~dV
   \leq M \left(R^d\phi(R^{-1}) + \int_{B_R} \phi(|\nabla\Q_L|) dV \right)
  \end{split}
 \end{equation}
 for a constant~$M>0$ that only depends on the characteristics of $\phi$ and on~$A$, $B$, $C$.
\end{lemma}
\begin{proof}
 As in~\cite[Lemma~2.2]{changyouwang}, we take a cut-off 
 $\eta\in C^\infty_{\mathrm{c}}(B_R)$ such that $\eta = 1$ on~$B_{R/2}$, $0\leq\eta\leq 1$ on~$B_R$, 
 $\abs{\nabla\eta}\lesssim 1/R$ and multiply the system by~$\nu(\Q_L)\eta^2$. 
 After integration by parts, and dropping the subscript~$L$, we get
 \begin{equation*}
  \frac{1}{L} \int_{B_R}
   \left( \frac{\partial f_B}{\partial Q_{ij}} + \frac{B}{3}|\Q|^2\delta_{ij}\right)\nu_{ij}(\Q)\eta^2~dV  =
   \underbrace{-\int_{B_R} \frac{\phi^\prime(|\nabla\Q|)}{|\nabla\Q|}\nabla\Q\cdot\nabla (\nu(\Q) \eta^2)}_{=:I}
 \end{equation*}
 Due to Lemma~\ref{lem:changyouwang1}, the left-hand side of this formula
 bounds the left-hand side of~\eqref{eq:p4} from above; therefore, it suffices to bound~$I$.
 By expanding~$\nabla (\nu(\Q) \eta^2)$ with the chain rule
 and using the fact that $|(\nabla\nu)(\Q)|\leq C$ for~$\dist(\Q, \NN)\leq\delta_0$, we obtain
 \begin{align*}
  \abs{I} \lesssim \int_{B_R} \left(\phi^\prime(|\nabla\Q|)|\nabla\Q| 
  + \frac{1}{R} \phi^\prime(\abs{\nabla\Q}) \right)
 \end{align*}
 For the first term in the right-hand side, we use that
 $\phi^\prime(t) t \sim\phi(t)$.
 For the second term, we apply the Young inequality~\eqref{Young} and Lemma~\ref{lem:Orliczproperties}:
 \begin{equation} \label{young}
  \frac{1}{R}\phi^\prime(\abs{\nabla\Q}) \leq
  \phi(R^{-1}) + \phi^*(\phi^\prime(\abs{\nabla\Q}))
  \lesssim \phi(R^{-1}) + \phi(\abs{\nabla\Q})
 \end{equation}
 where $\phi^*$ is defined in Section \ref{sect:notation}.
 Hence, the lemma follows.
\end{proof}

\begin{lemma}
\label{lem:changyouwang2}
 Let $\Q_L$ be a critical point of the modified LdG energy
 $I_{mod}$ in \eqref{eq:2}, on the unit ball~$B_1$, for a fixed~$L>0$. We assume that 
 \begin{equation} \label{closedness-bis}
  \dist\left(\Q_L, \NN \right) \leq \delta_0 \qquad \textrm{for all } x\in B_1
 \end{equation}
 where $\delta_0$ is sufficiently small for \eqref{eq:p3} to hold, and that
 \begin{equation} \label{Lipschitz}
  \Lambda := \sup_{B_1} \phi(\abs{\nabla \Q_L}) < +\infty.
 \end{equation}
 Then, for any integer~$q\geq 1$ and any~$0<\theta<1$, there exist
 $\delta_q \in \left[\theta^{q}, \, 1\right]$ and a positive constant $C_q>0$ 
 (depending on~$\Lambda$, $\theta$, $q$, $f_B$ and the characteristics of~$\phi$
 but not on~$L$) such that
 \begin{equation}
 \label{eq:p6}
 \int_{B_{\delta_q}} \left( \frac{1}{L} \abs{\Q_L-\Pi(\Q_L)} \right)^q dV \leq C_q
 \end{equation}
\end{lemma}

\begin{proof}
From Lemma \ref{lem:changyouwang1.5} and the bound~\eqref{Lipschitz},
this is true for $q=1$ and~$\delta_1 := 1$. 
By induction, we can assume that \eqref{eq:p6} is true for an integer $q\geq 2$. 
Then by Fubini's theorem, there exists $\delta_{q+1}\in \left(\theta\delta_{q}, \delta_q \right)$ such that 
\begin{equation}
\label{eq:p7}
 \int_{\partial B_{\delta_{q+1}}} \left(\frac{1}{L}
 \abs{\Q_L - \Pi(\Q_L)} \right)^q dV \leq C_{q+1}.
\end{equation}
We multiply both sides of the Euler-Lagrange equations \eqref{eq:EL} by 
$L^{-q}\left|\Q - \Pi(\Q)\right|^q \nu_{ij}(\Q)$ to get at the right-hand side
(dropping the subscript $L$ for brevity)
\begin{equation} \label{eq:p11}
 \begin{split}
  & \frac{1}{L^{q+1}} \int_{B_{\delta_{q+1}}} \left(\frac{\partial f_B}{\partial Q_{ij}}(\Q) 
       + \frac{B}{3}|\Q|^2\delta_{ij} \right) \nu_{ij}(\Q) \left| \Q - \Pi(\Q)\right|^q \\ 
  & \geq \frac{1}{L^{q+1}} \int_{B_{\delta_{q+1}}} \left|\Q - \Pi(\Q)\right|^{q+1} \\
 \end{split}
\end{equation}
where we have used the inequality \eqref{eq:p3}. 
Similarly, at the left-hand side we have 
\begin{equation} \label{eq:p12}
 \begin{split}
  & \int_{B_{\delta_{q+1}}} \frac{\partial}{\partial \x_k}\left[ \frac{2}{p}\psi^\prime\left(|\nabla \Q|^2\right)\Q_{ij,k} \right]
  \frac{1}{L^q}\left| \Q - \Pi(\Q)\right|^q \nu_{ij}(\Q) dV =  \\
  & \int_{\partial B_{\delta_{q+1}} }\frac{2}{p} \psi^\prime\left(|\nabla \Q|^2\right)\left(\frac{\Q_{ij,k}\x_k}{\delta_{q+1}}\right) 
  \frac{1}{L^q}\left| \Q - \Pi(\Q)\right|^q \nu_{ij}(\Q) d\sigma  \\ 
  &- \int_{B_{\delta_{q+1}}} \frac{2}{p}\psi^\prime\left(|\nabla \Q|^2\right)
  \Q_{ij,k}\frac{\partial}{\partial \x_k}\left\{\frac{1}{L^q}\left| \Q - \Pi(\Q)\right|^q \nu_{ij}(\Q) \right\}dV.
 \end{split}
\end{equation}
We estimate each integral on the right-hand side separately. 
The boundary integral can be estimated easily as shown below using the inequality 
\eqref{eq:p7}:
\begin{equation}\label{eq:p13}
 \begin{split}
  & \left|\int_{\partial B_{\delta_{q+1}} }\frac{2}{p} \psi^\prime\left(|\nabla \Q|^2\right)
  \left(\frac{\Q_{ij,k}\x_k}{\delta_{q+1}}\right) \frac{1}{L^q}\left| \Q - \Pi(\Q)\right|^q \nu_{ij}(\Q) d\sigma \right| \\
  & \leq \bar{c} \max_{B_1}\psi^\prime(|\nabla \Q|^2)\abs{\nabla\Q} C_{q+1}
  \leq \bar{c} \max_{B_1}\phi^\prime(|\nabla \Q|) C_{q+1}
 \end{split}
\end{equation}
for a positive constant $\bar{c}$ independent of $L$; $C_{q+1}$ has been defined
in~\eqref{eq:p7}. The right-hand side is bounded due to~\eqref{Lipschitz}. We then consider
\begin{equation} \label{eq:p14}
 \begin{split}
  & - \int_{B_{\delta_{q+1}}} \frac{2}{p}\psi^\prime\left(|\nabla \Q|^2\right) \Q_{ij,k}\frac{\partial}{\partial \x_k}\left\{\frac{1}{L^q}\left| \Q - \Pi(\Q)\right|^q \nu_{ij}(\Q) \right\}dV = \\
  & -\frac{2}{p}\int_{B_{\delta_{q+1}}}\psi^\prime\left(|\nabla \Q|^2\right) \Q_{ij,k} \frac{\partial \nu_{ij}(\Q)}{\partial \x_k}\frac{1}{L^q}\left| \Q - \Pi(\Q)\right|^q dV - \\
  & - \frac{2}{p}\int_{B_{\delta_{q+1}}}\psi^\prime\left(|\nabla \Q|^2\right) \Q_{ij,k} \nu_{ij}(\Q)\frac{q}{L^q}\left| \Q - \Pi(\Q)\right|^{q-1} \frac{\partial}{\partial \x_k}\left| \Q - \Pi(\Q)\right|dV.
 \end{split}
\end{equation}
It is relatively straightforward to see that
\begin{equation}
\label{eq:p15}
\left| \int_{B_{\delta_{q+1}}} \frac{2}{p}
\psi^\prime\left(|\nabla \Q|^2\right) \Q_{ij,k} \frac{\partial \nu_{ij}(\Q)}{\partial \x_k}\frac{1}{L^q}\left| \Q - \Pi(\Q)\right|^q dV \right|
 \lesssim \Lambda C_q 
\end{equation}
where we have used the hypotheses \eqref{Lipschitz} and~\eqref{eq:p7}. It remains to note that \cite{changyouwang}
\[
 \Q_{ij,k} \nu_{ij} \frac{\partial}{\partial \x_k}\left| \Q - \Pi(\Q)\right| 
 = \Q_{ij,k} \nu_{ij} \Q_{\alpha\beta, k}\nu_{\alpha\beta} \geq 0
\]
so that
\[
 \frac{2}{p}\int_{B_{\delta_{q+1}}}\psi^\prime\left(|\nabla \Q|^2\right) \Q_{ij,q} \nu_{ij}(\Q)\frac{q}{L^q}\left| \Q - \Pi(\Q)\right|^{q-1} \frac{\partial}{\partial \x_k}\left| \Q - \Pi(\Q)\right|dV \geq 0
\]
and the conclusion of the lemma follows. $\Box$
\end{proof}

We can now turn to the proof of Proposition~\ref{prop:uniform1}.

\begin{proof}[of Proposition~\ref{prop:uniform1}] 
 The proof follows from Lemma $7$ of \cite{majumdarzarnescu} 
and Lemma $2.3$ of \cite{changyouwang}. 
There is an additional step towards the end of the proof which was not needed/considered in either \cite{majumdarzarnescu} or \cite{changyouwang}.

We can take $\x^* = \mathbf{0}$ without loss of generality. 
Choose $r_1^L\in (r/2, \, r)$ and $\x_L \in B_{r_1^L}$ such that
\begin{equation}
\label{eq:e1bis}
\max_{r/2\leq s\leq r} (r - s)^p \max_{B_s} e_{L}(\Q_{L}) = (r - r_1^L)^p \max_{B_{r_1^L}} e_L(\Q_{L}) 
\end{equation}
and
\begin{equation}
\label{eq:e2bis}
\max_{B_{r_1^L}} e_L(\Q_{L})  = e_L (\Q_{L})(\x_L) =: K_L^p.
\end{equation}
Such $r_1^L$ and $\x_L$ exist, because $\Q_{L}\in C^{1,\alpha}_{\mathrm{loc}}(\Omega)$ 
(Proposition~\ref{prop:C1,a}) and hence $e_{L}(\Q_{L})$ is continuous.
Since $\frac{r}{2} < \frac{r + r_1^L}{2} < r$, by definition of~$r^L_1$ we have
\begin{equation*}
\label{eq:e3}
\left( r - \frac{r + r_1^L}{2} \right)^p \max_{B_{(r + r_1^L)/2}} e_{L}(\Q_{L}) \leq (r - r_1^L )^p K_L^p
\end{equation*}
so that
\begin{equation}
\label{eq:e4}
\max_{B_{(r + r_1^L)/2}} e_{L}(\Q_{L}) \leq 2^p K_L^p.
\end{equation}

Next, we set $r_2^L := \frac{r - r_1^L}{2}K_L$ and $\bar{L} :=  L K_L^p$,
with the scaled map
\begin{equation*}
\label{eq:e5}
\v_{\bar{L}}(\x) := \Q_{L}\left( \x_L + \frac{\x}{K_L} \right) \quad 
\textrm{for } \x\in B_{r_2^L},
\end{equation*}
the scaled elastic modulus
\begin{equation}
 \tilde{\phi}_{K_L}(t) := K_L^{-p} \phi(K_L t) \qquad\textrm{for } t\geq 0
\end{equation}
and the scaled energy density
\begin{equation*}
\bar{e}_{\bar{L}}(\v_{\bar{L}}) := \tilde{\phi}_{K_L}(|\nabla\v_{\bar{L}}|) + \bar{L}^{-1} f_B(\v_{\bar{L}}).
\end{equation*}
We compute that $|\nabla\Q_L| = K_L|\nabla\v_{\bar{L}}|$ and 
$e_{L}(\Q_{L}) = K_L^{p}\bar{e}_{\bar{L}}(\v_{\bar{L}})$, so 
we have from \eqref{eq:e2bis} and \eqref{eq:e4} that
\begin{equation}
\label{eq:e6}
\max_{B_{r^L_2}} \bar{e}_{\bar{L}}\left( \v_{\bar{L}}\right) \leq 2^p,
\qquad \bar{e}_{\bar{L}}(\v_{\bar{L}})(\mathbf{0}) = 1.
\end{equation}
Further, from the Euler-Lagrange equations \eqref{eq:EL}, we have that
$\v_{\bar{L}}$ is a solution of
\begin{equation}\label{eq:e7}
\begin{split}
\div\left(\frac{\tilde{\phi}_{K_L}(\abs{\nabla\v_{\bar{L}}})}{\abs{\nabla \v_{\bar{L}}}}\nabla \v_{\bar{L}} \right) 
= \frac{1}{\bar{L}}\left[ -A \v_{\bar{L}} - B \left(\v_{\bar{L}} \v_{\bar{L}} - |\v_{\bar{L}}|^2 \frac{\mathbf{I}}{3} \right) + C|\v_{\bar{L}}|^2 \v_{\bar{L}} \right].
\end{split}
\end{equation}
(in fact, $\v_{\bar{L}}$ is a minimizer of the associated functional,
which is obtained from~$I_{LdG}$ by scaling).

We now consider two possibilities.
If $r_2^L \leq 1$ then, 
by choosing~$s = r/2$ in Eq.~\eqref{eq:e1bis}, we deduce that
\begin{equation}
\label{eq:e8}
\left(\frac{r}{2} \right)^p \max_{B_{r_1^L}} e_L(\Q_{L})
\leq \left( r - r_1^L \right)^p K_L \stackrel{r^L_2\leq 1}{\leq} 2^p
\end{equation}
verifying the pointwise bound on the energy density in (\ref{eq:u3}).
The second case is $r_2^L > 1$. 
We claim that if $r_2^L>1$, then there exists a positive constant $C_0$ independent of $L$ such that
\begin{equation}
\label{eq:e9}
C_0 \leq \int_{B_1} \bar{e}_{\bar{L}} (\v_{\bar{L}}).
\end{equation}
Assuming that \eqref{eq:e9} holds, we will get the required contradiction 
with \eqref{eq:u2}, by appealing to the monotonicity of the normalised energy in Lemma \ref{lem:3}. 
Indeed, since $r_2^L > 1$ by assumption, we have~$K_L^{-1} < r$ and from~\eqref{eq:e9} we deduce, by scaling,
\begin{equation*}
 \begin{split}
 C_0 \leq K_L^{-p+d}\int_{B_{K_L^{-1}}} e_{L} (\Q_L) &\stackrel{\eqref{eq:G1}}{\leq}
 r^{p-d}\int_{B_r} e_{L} (\Q_L) + Mr^p 
 \stackrel{\eqref{eq:u2}}{\leq} \varepsilon + Mr_0^p
 \end{split}
\end{equation*}
(we have used that 
$r\leq r_0$). Thus, we obtain a contradiction if we choose~$\varepsilon$, $r_0$ small enough.

The rest of the proof is dedicated to proving the inequality (\ref{eq:e9}).
By Equation~\eqref{eq:e7} and Theorem~\ref{thm:subsol}, we know that 
$w_{\bar{L}} := \tilde{\phi}_{K_L}(|\nabla\v_{\bar{L}}|)$ is a subsolution of an elliptic problem, namely
\begin{equation} \label{subharmonic}
 -\div\left(\mathbf{G}_{\bar{L}} \nabla w_{\bar{L}}\right) \leq -\frac{1}{\bar{L}} 
 \frac{ \partial^2 f_B(\v_{\bar{L}})}{\partial Q_{ij}\partial Q_{hk}} 
 \partial_\ell\v_{\bar{L}ij}\partial_\ell\v_{\bar{L}hk} \quad \textrm{on } B_1
\end{equation}
where the tensor field~$\mathbf{G}_{\bar{L}} = \mathbf{G}_{\bar{L}}(\nabla\v_{\bar{L}})$
is bounded and elliptic:
\begin{equation} \label{ellipticity}
 \alpha_0\abs{\xi}^2 \leq \mathbf{G}_{\bar{L}}\xi\cdot\xi \leq \alpha_1\abs{\xi}^2
 \qquad \textrm{for any } \xi\in\R^d.
\end{equation}
Although the function~$\tilde{\phi}_{K_L}$ does depend on~$K_L$,
the constants~$\alpha_0$, $\alpha_1$ are independent of~$K_L$. Indeed, by
Proposition~\ref{prop:subsol} we know that~$\alpha_0$, $\alpha_1$ only depend on the ratio
\[
 \frac{\tilde{\phi}^{\prime\prime}_{K_L}(t)t - \tilde{\phi}^\prime_{K_L}(t) }{\tilde{\phi}^\prime_{K_L}(t)}
 = \frac{\phi^{\prime\prime}(K_Lt)K_Lt - \phi^\prime(K_Lt)}{\phi^\prime(K_Lt)}
\]
which is bounded form above and below, independently of~$K_L$, 
thanks to Assumption~\eqref{hp:Delta2prime}. Now, let~$q>d$ be fixed, and 
let us set~$x^+ := \min\{x, \, 0\}$ for~$x\in\R$. We claim that
\begin{gather} 
 \int_{B_{1/2}}\left(\left(-\frac{1}{\bar{L}} 
 \frac{ \partial^2 f_B(\v_{\bar{L}})}{\partial Q_{ij}\partial Q_{hk}} \label{Lq1}
 \partial_\ell\v_{\bar{L}ij}\partial_\ell\v_{\bar{L}hk}\right)^+\right)^q \leq C_q \\
 \int_{B_{1/2}} \left(\frac{1}{\bar{L}} \abs{(\nabla_{\Q}f_B)(\v_{\bar{L}})}\right)^q \leq C_q \label{Lq2}
\end{gather}
for some~$C_q$ independent on~$L$, $K_L$.
Let~$\delta_0>0$ be given by Lemma~\ref{lem:changyouwang1}, and let
\begin{equation} \label{eq:e10}
 \gamma := \inf\left\{f_B(\Q)\colon\Q\in S_0, \ \dist(\Q, \, \NN)\geq\delta_0\right\} \, .
\end{equation}
Since~$f_B(\Q)>0$ if~$\Q\notin\NN$ and $f_B(\Q)\to+\infty$ as~$|\Q|\to+\infty$,
we have that~$\gamma>0$. We now consider two cases separately.

\noindent \textbf{Proof of~\eqref{Lq1}, \eqref{Lq2} --- Case I:} $\bar{L}\geq 2^{-p}\gamma$.
This is straightforward because in this case, the left-hand sides of~\eqref{Lq1}, \eqref{Lq2}
are uniformly bounded in terms of~$\gamma$, $A$, $B$, and~$C$ (since we have the maximum principle, Lemma~\ref{lemma:max_principle}).

\noindent \textbf{Proof of~\eqref{Lq1}, \eqref{Lq2} --- Case II:} $\bar{L}<2^{-p}\gamma$.
From \eqref{eq:e6}, the energy densities are uniformly bounded and in particular
\[
 f_B( \v_{\bar{L}}) \leq 2^p\bar{L} < \gamma \quad \textrm{on } B_1\subseteq B_{r^L_2},
\]
so, because of our choice of~$\gamma$ in~\eqref{eq:e10}, we have 
\[
 \dist(\v_{\bar{L}}, \, \NN)\leq\delta_0 \qquad\textrm{on } B_1. 
\]
Then, we can consider the projection~$\Pi(\v_{\bar L})$ onto~$\NN$.
Since the potential~$f_B$ is smooth, and in particular the second derivatives of~$f_B$
are Lipschitz-continuous, we have
\[
 - \frac{ \partial^2 f_B(\v_{\bar{L}})}{\partial Q_{ij}\partial Q_{hk}} 
 \partial_\ell\v_{\bar{L}ij}\partial_\ell\v_{\bar{L}hk} 
 \leq - \frac{ \partial^2 f_B(\Pi(\v_{\bar{L}}))}{\partial Q_{ij}\partial Q_{hk}} 
 \partial_\ell\v_{\bar{L}ij}\partial_\ell\v_{\bar{L}hk} + M\abs{\v_{\bar{L}} - \Pi(\v_{\bar{L}})}\abs{\nabla\v_{\bar{L}}}^2
\]
for some constant~$M$ that only depends on the coefficients of~$f_B$.
Since $\Pi(\v_{\bar{L}})$ belongs to the minimizing manifold of~$f_B$,
the Hessian matrix of~$f_B$ at~$\Pi(\v_{\bar{L}})$ is positive semi-definite 
and hence the first term in the right-hand side is non-positive. 
Recalling that $|\nabla\v_{\bar{L}}|$ is bounded by~\eqref{eq:e6}, we obtain
\[
 \begin{split}
  &\int_{B_{1/2}}\left(\left(-\frac{1}{\bar{L}} 
  \frac{ \partial^2 f_B(\v_{\bar{L}})}{\partial Q_{ij}\partial Q_{hk}} 
  \partial_\ell\v_{\bar{L}ij}\partial_\ell\v_{\bar{L}hk}\right)^+\right)^q \, \d x \\
  &\quad \leq \frac{M^q}{L^q}\int_{B_{1/2}}\abs{\v_{\bar{L}} - \Pi(\v_{\bar{L}})}^q \abs{\nabla\v_{\bar{L}}}^{2q}
  \stackrel{\eqref{eq:e6}}{\lesssim} \frac{M^q}{L^q}\int_{B_{1/2}}\abs{\v_{\bar{L}} - \Pi(\v_{\bar{L}})}^q
 \end{split}
\]
and the right-hand side is bounded by Lemma~\ref{lem:changyouwang2}, so~\eqref{Lq1} follows.
The proof of~\eqref{Lq2} follows by a similar reasoning, using the fact that $\nabla_{\Q}f_B(\Pi(\Q)) = 0$ because~$\Pi(\Q)\in\NN$ is a minimum point of~$f_B$.

\noindent \textbf{Proof of~\eqref{eq:e9}:}
Because of~\eqref{eq:e6}, we must either have $\bar{L}^{-1} f_B(\v_{\bar{L}}(\mathbf{0}))\geq 1/2$
or $w_{\bar{L}}(\mathbf{0}) = \tilde{\phi}_{K_L}(|\nabla\v_{\bar{L}}(\mathbf{0})|)\geq 1/2$.
In case $\bar{L}^{-1} f_B(\v_{\bar{L}}(\mathbf{0}))\geq 1/2$, we compute using the chain rule
\begin{equation} \label{uniform-W1q}
 \abs{\nabla_{\x}\left(\bar{L}^{-1} f_B(\v_{\bar{L}})\right)}
 \leq \bar{L}^{-1} \abs{(\nabla_{\Q}f_B)(\v_{\bar{L}})}\abs{\nabla\v_{\bar{L}}}.
\end{equation}
Now, $\bar{L}^{-1}(\nabla_{\Q}f_B)(\v_{\bar{L}})$ is uniformly bounded in~$L^q(B_{1/2})$
due to~\eqref{Lq2}, while $\nabla\v_{\bar{L}}$ is uniformly bounded in~$L^\infty(B_1)$ 
due to~\eqref{eq:e6}. It follows that $\bar{L}^{-1} f_B(\v_{\bar{L}})$
is uniformly bounded in~$W^{1,q}(B_{1/2})$ with~$q>d$ and hence, by Sobolev embedding,
$\bar{L}^{-1} f_B(\v_{\bar{L}})$ is $(1-d/q)$-H\"older continuous on~$B_{1/2}$, with uniform bound on
the H\"older norm. Therefore, \eqref{eq:e9} follows immediately if we have
$\bar{L}^{-1} f_B(\v_{\bar{L}}(\mathbf{0}))\geq 1/2$.

Finally, it only remains to consider the case $w_{\bar{L}}(\mathrm{0})\geq 1/2$.
In this case, we use the elliptic inequality~\eqref{subharmonic}, together with the
ellipticity bounds~\eqref{ellipticity} and the bound on the right hand side~\eqref{Lq1}.
By applying the theory for elliptic equations (see e.g.~\cite[Theorem~8.3]{GilbargTrudinger}),
we deduce that
\begin{equation} \label{GT1}
 \frac{1}{2} \leq w_{\bar{L}}(\mathbf{0}) \leq \sup_{B_{R/4}} w_{\bar{L}} \leq 
 C_1 R^{-d/q} \|w_{\bar{L}}\|_{L^q(B_R)} + C_2R^{2(1-d/q)}
\end{equation}
for any~$R\in (0, \, 1/2)$ and for some constants~$C_1$, $C_2$
that are independent of~$L$. We choose
\[
 R := \min\left\{\frac{1}{2}, \, 
 \left(\frac{1}{4C_2}\right)^{\frac{1}{2(1-d/q)}}\right\} \!, 
\]
so that $C_2R^{2(1-d/q)}\leq 1/4$ and the second term in the right-hand side of~\eqref{GT1} can be absorbed into 
the left-hand side. Then, using the interpolation inequality, we obtain
\[
 \begin{split}
  \frac{1}{4} \leq C_1 R^{-d/q}\|w_{\bar{L}}\|_{L^q(B_R)} 
  &\leq C_1 R^{-d/q} \|w_{\bar{L}}\|^{1/q}_{L^1(B_R)} \|w_{\bar{L}}\|^{1-1/q}_{L^\infty(B_R)} \\
  &\stackrel{\eqref{eq:e6}}{\leq} C_1 R^{-d/q} 2^{p-p/q} \|w_{\bar{L}}\|^{1/q}_{L^1(B_R)},
 \end{split}
\]
whence~\eqref{eq:e9} follows.
\end{proof}

\subsection{Uniform $C^{1,\alpha}$ estimates 
and the proof of Theorem~\ref{th:convergence}}

Proposition~\ref{prop:uniform1} provides a uniform bound, independent of~$L$, 
in the regions where the energy is small (i.e., away from the singularities of the 
limiting harmonic map). 
In this section, we deduce a uniform $C^{1,\alpha}$ bound on~$\Q_L$ 
from the uniform bound on~$e_L(\Q_L)$; this will allow us to 
complete the proof of Theorem~\ref{th:convergence}.
To this end, given a ball $B(\x_0, R)\subset \Omega$ 
and a minimizer~$\Q_L$ of our functional, 
we consider again the $\phi$-harmonic replacement of $\Q_L$ inside~$B(\x_0, R)$,
i.e. the unique solution~$\P$ of 
\begin{equation} \label{frozen-bis}
 \div\left(\frac{\phi^\prime(|\nabla \P|)}{|\nabla\P|}\nabla\P\right) = 0 \quad 
 \textrm{in } B(x_0, R), \qquad \P = \bar{\Q}
 \quad \textrm{on } \partial  B(\x_0, R).
\end{equation}

\begin{lemma} \label{lemma:excess-difference}
 Let $\Q_L$ be a minimizer of the modified LdG energy
 $I_{mod}$ in \eqref{eq:2}, on the ball~$B(\x_0, 2R)$, 
 for a fixed (but arbitrary)~$L>0$. We assume that 
 \begin{equation} \label{density-bound}
  \Lambda := \sup_{B(\x_0, 2R)}
  \left(\phi(\abs{\nabla \Q_L}) + \frac{1}{L} f_B(\Q_L) \right)< +\infty.
 \end{equation}
 Let $\P$ be the solution of~\eqref{frozen-bis}.
 Then, 
 there exists a positive constant~$C_{\Lambda}$ 
 (depending on~$\Lambda$, but not on~$L$) such that
 \begin{equation*}
 \mI{B(\x_0, R)}|\V(\nabla \Q_L)-\V(\nabla \P)|^2 dx \leq C_{\Lambda} R. 
\end{equation*}
\end{lemma}
\begin{proof}
 By Equation~\eqref{frozen_excess} the proof of
 Proposition~\ref{lemma:frozen_excess}, we know that
 \begin{equation*}
  \mI{B_R}|\V(\nabla\Q_L)-\V(\nabla \P)|^2 \d x
  \le \frac{1}{L} \mI{B_R} \abs{\nabla_\Q f_B(\Q_L)} | \Q_L-\P| \d x .
 \end{equation*}
 Under the assumption~\eqref{density-bound}, we can apply
 Lemma~\ref{lem:changyouwang1.5} to obtain an uniform (i.e. $L$-independent)
 $L^q$-bound for $L^{-1}\nabla_\Q f_B(\Q_L)$, 
 where~$q\in[1, \, +\infty)$ is arbitrary. Then, the H\"older inequality gives
 \begin{equation*}
  \mI{B_R}|\V(\nabla\Q_L)-\V(\nabla \P)|^2 \d x \le C_\Lambda
  \left(\mI{B_R} | \Q_L-\P|^{\frac{q}{q-1}} \d x \right)^{\frac{q-1}{q}} .
 \end{equation*}
 The assumption~\eqref{density-bound} also implies that $\Q_L$ 
 is Lipschitz continuous, and its Lispchitz constant is bounded in terms 
 of~$\Lambda$. Then, we can use the convex hull property,
 exactly as in the proof of~\eqref{convexhull}, to check that 
 $\|\Q_L - \P\|_{L^\infty(B_R)} \leq C_\Lambda R$.
 Therefore, the lemma follows.
\end{proof}

The $C^{1,\alpha}$-bound for~$\Q_L$ is now obtained
by repeating verbatim the arguments of Proposition~\ref{prop:C1,a} in Section~\ref{sect:splitting}.
Since Lemma~\ref{lemma:excess-difference} gives a $L$-independent bound, 
we are now able to deduce regularity estimates that do not depend on~$L$. 
As a result, we obtain the

\begin{lemma} \label{lemma:uniformC1alpha}
 Let $\Q_L$ be a Lipschitz minimizer of functional~\eqref{eq:2}, 
 on the ball $B(\x_0, R)$. Let~$\Lambda$ be defined as in~\eqref{density-bound}.
 Then, there exists~$\alpha\in (0, \, 1)$ (only depending on the characteristics of~$\phi$)
 and a constant~$C_{\Lambda, R}$, depending on $\Lambda$, $R$ but not on~$L$, such that 
 \begin{equation*} 
  \|\nabla \Q_L\|_{C^\alpha(B(\x_0, R/4))}  \leq C_{\Lambda,R}.
 \end{equation*}
\end{lemma}

We are now ready to prove Theorem~\ref{th:convergence}.

\begin{proof}[of Theorem~\ref{th:convergence}]
 Let~$\Q_L$ be a minimizer of $I_{mod}$, in the class defined by~\eqref{eq:3}. By Lemma~\ref{lem:3}
 we know that, up to a non relabelled subsequence, $\Q_L\rightharpoonup\Q$ in~$W^{1,\phi}$ 
 as~$L\to 0$, where $\Q_0$ is a minimizing harmonic map. The set
 \[
  S[\Q_0] := \left\{\x\in\Omega\colon \liminf_{\rho\to 0}
  \rho^{p-d}\int_{B(\x, \rho)} \phi(|\nabla\Q_0|)> 0\right\}
 \]
 is closed, due to the monotonicity formula (Lemma~\ref{lem:mon}), and moreover we have
 $\H^{d-p}(S[\Q_0])= 0$ , see \cite{Giusti}.
 Let~$\varepsilon > 0$ be given by Proposition~\ref{prop:uniform1}.
 For a fixed~$\x_0\in\Omega\setminus S[\Q_0]$,
 we can find a radius $R = R(\x_0) >0$ such that
 $R^{p-d}\int_{B(\x_0, R)} \phi(|\nabla\Q_0|)\leq\varepsilon/2$.
 Then, due to Lemma~\ref{lem:3}, we have 
 \[
  R^{p-d}\int_{B(\x_0, R)} e_L(\Q_L) \leq\varepsilon
 \]
 for any~$L$ small enough. We can then apply Proposition~\ref{prop:uniform1}
 and deduce that $e_L(\Q_L)\leq C_R$ on~$B(\x_0, R/2)$, for some constant~$C_R$
 that depends on~$R$ but not on~$L$. Finally, we apply Lemma~\ref{lemma:uniformC1alpha}
 to obtain the uniform bound $\|\Q_L\|_{C^{1,\alpha}(B(\x_0, R/8))}\leq C_R$.
 Thanks to this uniform bound and to Ascoli-Arzel\`a theorem, we deduce at once 
 that $\Q_0\in C^{1,\alpha}(B(\x_0, R/8))$ and that $\Q_L\to\Q_0$, $\nabla\Q_L\to\nabla\Q_0$
 uniformly on~$B(\x_0, R/8)$. This completes the proof of the theorem.
\end{proof}

\section{Biaxiality in the low temperature limit}

In this section, we show the biaxial character of bidimensional defect cores, if the temperature is low enough. Throughout the section, we assume that $\Omega$ is a bounded, smooth domain in~$\R^2$.
To simplify the analysis, we take $\phi(t) := t^{p}/p$ so that the elastic energy density reduces to~$\frac{1}{p}|\nabla\Q|^p$. 
This is consistent with our assumptions~\eqref{hp:first}--\eqref{hp:last} and, in this section,
we are only interested in regions of large gradients.
Therefore, we do not expect that this simplification should affect the qualitative conclusions of the analysis.

We introduce the rescaled temperature
\begin{equation} \label{t}
 t := \frac{27|A|C}{B^2} > 0 
\end{equation}
and we are interested in the limit as~$t\to +\infty$.
After a suitable non-di\-men\-sio\-na\-li\-sa\-tion, along the lines of \cite{HenaoMajumdarPisante}, the modified Landau-de Gennes free energy functional~\eqref{eq:2} reduces to
\begin{equation} \label{rescaled-LdG}
 F_t[\Q] := \int_{\Omega} \left\{\frac{\alpha}{p}\abs{\nabla\Q}^p + T(t) \left(1 - |\Q|^2\right)^2 + H(t) g(\Q) \right\} \d V,
\end{equation}
where
\[
 \alpha := \left(\frac{3}{2}\right)^{1-p/2}\bar{L}D, \qquad
 \bar{L} := \frac{27CL}{2D^2B}
\]
and~$D$ is a typical length of the domain, 
\[
 T(t) := \frac{t}{8} s_+^{2-p}, \qquad H(t) := \frac{3 + \sqrt{9+8t}}{32} s_+^{2-p}
\]
and
\[
 s_+ = \frac{B + \sqrt{B^2 + 24 A C}}{4C} = \frac{3 + \sqrt{9 + 8t}}{12C/B}
\]
is the optimal value of the scalar order parameter. Finally,
\[
g(\Q) := 1 + 3|\Q|^4 - 4\sqrt{6}\tr\Q^3
\]
is a potential that penalizes biaxiality (it is straightforward to verify 
that $g(\mathbf{Q})$ is minimized by uniaxial tensors of unit norm).

Since~$s_+$ grows as~$t^{1/2}$ as~$t\to+\infty$,
we have the asymptotic estimates
\begin{equation} \label{H,Tscalings}
 T(t)\sim t^{2-p/2}, \qquad H(t)\sim t^{3/2 - p/2} \qquad \textrm{for } t\to+\infty.
\end{equation}
Therefore, both the parameters~$T(t)$ and~$H(t)$ diverge as~$t\to+\infty$, but the penalization associated with deviation from unit norm is stronger than the one associated with biaxiality.

Our main result of this section is the following.

\begin{theorem} \label{th:no-isotropy}
Let~$\Q_t$ be a minimizer of the functional~$F_t$, in the admissible class~$\mathcal{A}$ defined by~\eqref{eq:3}. For any~$\delta > 0$ there exists a positive number~$t_0(\delta)$ such that, if $t\geq t_0(\delta)$ and~$\x\in\Omega$ satisfies $\dist(\x, \, \partial\Omega)\geq\delta$, then $\Q_t(\x)\neq\mathbf{0}$.
\end{theorem}

This result guarantees that, if the temperature (measured in dimensionless units) is low enough, 
then the minimizer~$\Q_t$ does not possess any isotropic point, except possibly in
a neighbourhood of the boundary. For a quadratic energy, i.e. when~$p=2$, it is already known that in the low temperature regime, minimizers do not have isotropic points, and this holds both in two-dimensional \cite{Canevari2D,DiFrattaetal} and three-dimensional domains \cite{contreraslamy,HenaoMajumdarPisante}. 
While, on the one hand, the Landau-de Gennes potential favours biaxial phases over the isotropic one when the temperature is low, 
on the other hand isotropic defect cores such as the radially symmetric, uniaxial hedgehog are less heavily penalized by a subquadratic elastic energy, compared to the quadratic case.

In contrast with the quadratic case, we are not able to exclude the presence 
of isotropic points in a neighbourhood of the boundary.
This boundary layer is related to the fact that it is difficult 
to obtain boundary regularity for the $p$-Laplace equation.

The absence of isotropic points means that biaxial escape takes place in the defect cores and indeed, the presence of biaxiality can be deduced from Theorem~\ref{th:no-isotropy} by topological arguments.  
Given a matrix~$\Q\in S_0$, we denote its eigenvalues by~$\lambda_{\max}(\Q)\geq\lambda_{\textrm{mid}}(\Q)\geq\lambda_{\min}(\Q)$. We also introduce the biaxiality parameter
\[
 \beta^2(\Q) = 1 - 6\frac{\left(\textrm{tr}\,\Q^3\right)^2}{\left(\textrm{tr}\,\Q^2\right)^3}.
\]
We recall that the biaxiality parameter satisfies $0\leq\beta^2(\Q)\leq 1$ and that~$\beta^2(\Q) = 0$ if and only if~$\Q$ is uniaxial.

\begin{corollary} \label{cor:biaxiality} 
Suppose that the boundary datum $\Q_b$ is topologically non-trivial, i.e. there is no continuous map $\Q\colon\overline{\Omega}\to\NN$ such that $\Q = \Q_b$ on~$\partial\Omega$.
Let~$\Q_t$ be a minimizer of~$F_t$ in the admissible class~$\mathcal{A}$ defined by~\eqref{eq:3}.
Then, for any~$\delta > 0$ there exists a positive number~$t_0(\delta)$ with the following property: if $t\geq t_0(\delta)$ and~$\lambda_{\max}(\Q_t(\x))>\lambda_{\mathrm{mid}}(\Q_t(\x))$ for any~$\x\in\Omega$ with~$\dist(\x, \, \partial\Omega)\leq\delta$, then
\[
 \max_{\x\in\Omega} \beta^2(\Q_t(\x)) = 1.
\]
\end{corollary}

We prove Theorem~\ref{th:no-isotropy} by contradiction, following the strategy in~\cite{contreraslamy}. Suppose that there exists a number~$\delta > 0$ a sequence~$t_j\nearrow+\infty$ and a sequence of points~$(\x_{t_j})$ in~$\Omega$ such that
\begin{equation} \label{xt}
 \dist(\x_{t_j}, \, \partial\Omega)\geq\delta, \qquad \Q_{t_j}(\x_{t_j}) = 0.
\end{equation}
From now on, we omit the subscript~$j$ and write~$t$, $\x_t$ instead of~$t_j$, $\x_{t_j}$.

Let $\delta_t := \delta t^{2/p - 1/2}\to+\infty$ as~$t\to+\infty$. 
We define the blown-up map
\begin{equation} \label{blownup}
 \bar{\Q}_t(\x) := \Q(\x_t + t^{1/2 - 2/p}\x) \qquad \textrm{for } \x\in B_{\delta_t}.
\end{equation}
The map~$\bar{\Q}_t$ minimizes the rescaled functional
\[
 \bar{F}_t(\bar{\Q}) := \int_{B_{\delta_t}} \left\{\frac{\alpha}{p}\abs{\nabla\bar{\Q}}^p + \bar{T}(t) \left(1 - |\bar{\Q}|^2\right)^2 + \bar{H}(t) g(\bar{\Q}) \right\} \d V
\]
subject to its own boundary conditions. Here, $\bar{T}(t)$ and~$\bar{H}(t)$ are defined by
\[
 \bar{T}(t) := t^{p/2 - 2}T(t), \qquad \bar{H}(t) := t^{p/2 - 2}H(t).
\]
Thanks to~\eqref{H,Tscalings}, we see that $\bar{H}(t)\sim t^{-1/2}\to 0$ as~$t\to +\infty$,
while $\bar{T}(t)$ is a bounded function of~$t$ and, in fact, it converges to a finite limit as~$t\to+\infty$:
\begin{equation} \label{beta}
 \lim_{t\to+\infty}\bar{T}(t) = 2^{p/2 - 4} \cdot 3^{p - 2} 
 \left(\frac{C}{B}\right)^{p - 2} =: \gamma>0.
\end{equation}
Moreover, $\bar{\Q}_t$ is a solution of the Euler-Lagranges equations:
\begin{equation} \label{EL-barQt}
 -\alpha\div\left(\abs{\nabla\bar{\Q}_t}^{p-2}\nabla\bar{\Q}_t\right) = 
 4\bar{T}(t) (1 - |\bar{\Q}_t|^2)\bar{\Q}_t + 12\bar{H}(t) 
 (\sqrt{6}\bar{\Q}_t^2 - |\bar{\Q}_t|^2\bar{\Q}_t) 
\end{equation}
on the ball~$B_{\delta_t}$.
A straightforward modification the maximum principle arguments in Lemma~\ref{lemma:max_principle} shows that the solutions of this system of equations satisfy $|\bar{\Q}_t|\leq 1$ pointwise a.e. on~$B_{\delta_t}$. Then, the right-hand side of~\eqref{EL-barQt} is bounded in~$L^\infty$, uniformly in the parameter~$t>0$. The regularity theory for $p$-Laplace systems implies that
\begin{equation} \label{holder_gradient}
 \|\nabla\bar{\Q}_t\|_{C^\theta(B_R)} \leq C_R \qquad \textrm{for any } R > 0
 \textrm{ and any } t \textrm{ large enough,}
\end{equation}
for some uniform~$\theta\in (0, \, 1)$ and some constant~$C_R$ that depends on~$R$ but not on~$t$. As a consequence, up to extraction of a subsequence, the maps $\bar{\Q}_t$ converge locally uniformly to a continuous map~$\bar{\Q}_\infty\colon\R^2\to S_0$. 

\begin{lemma} \label{lemma:blowup}
 For any~$R>0$, the map~$\bar{\Q}_\infty$ minimizes the functional
 \[
  \bar{F}_\infty(\Q; \, B_R) := \int_{B_R} \left\{ \frac{\alpha}{p}\abs{\nabla\Q}^p + \gamma \left(1 - |\Q|^2\right)^2\right\} \d V
 \]
 among all maps~$\Q\in W^{1,p}(B_R; \, S_0)$ such that $(1 - |\Q|^2)^2\in L^1(B_R)$ and $\Q = \bar{\Q}_\infty$ on~$\partial B_R$. Recall that~$\gamma$ has been defined in~\eqref{beta}. 
 Moreover, $|\bar{\Q}_\infty|\leq 1$ on~$\R^2$ and there holds $\bar{\Q}_\infty(\mathbf{0}) = \mathbf{0}$.
\end{lemma}
\begin{proof}
 By the locally uniform convergence $\bar{\Q}_t\to\bar{\Q}_\infty$, we immediately see 
 that $|\bar{\Q}_\infty|\leq 1$ on~$\R^2$, $\bar{\Q}_\infty(\mathbf{0}) = \mathbf{0}$.
 Let~$\Q\in W^{1,p}(B_R; \, S_0)$ be an admissible competitor for~$\bar{\Q}_\infty$,
 i.e. a map such that $(1 - |\Q|^2)^2\in L^1(B_R)$ and $\Q = \bar{\Q}_\infty$ on~$\partial B_R$.
 By a truncation argument, we can assume w.l.o.g. that $\Q\in L^\infty(B_R)$. 
 If~$t$ is large enough, so that $B_R\subseteq B_{\delta_t}$, then
 $\Q  + \bar{\Q}_t - \bar{\Q}_\infty$ is an admissible competitor for~$\bar{\Q}_t$ and we have 
 $\bar{F}_t(\bar{\Q}_t) \leq \bar{F}_t(\Q  + \bar{\Q}_t - \bar{\Q}_\infty)$.
 Moreover, thanks to the uniform bound~\eqref{holder_gradient} and to Ascoli-Arzel\`a theorem,
 we deduce that $\nabla\bar{\Q}_t\to \nabla\bar{\Q}_\infty$ locally uniformly in~$\R^2$.
 We can hence pass to the limit as~$t\to+\infty$ in the inequality
 $\bar{F}_t(\bar{\Q}_t) \leq \bar{F}_t(\Q  + \bar{\Q}_t - \bar{\Q}_\infty)$ and deduce that 
 $\bar{F}_\infty(\bar{\Q}_\infty; \, B_R)\leq \bar{F}_\infty(\Q; \, B_R)$.
\end{proof}

\begin{lemma} \label{lemma:bound_monotonicity}
 For any~$R>0$, there holds
 \[
  \bar{F}_\infty(\bar{\Q}_\infty; \, B_R) \leq C R^{2-p}
 \]
 for some constant~$C$ that does not depend on~$R$. 
\end{lemma}
\begin{proof} 
 We can reproduce the arguments of~\cite[Lemma~3.6]{HenaoMajumdarPisante}
 using the monotonicity formula given by Lemma~\ref{lem:mon}.
\end{proof}

For any~$R>0$, we define the map $\u_R\colon B_1\to S_0$ by 
\begin{equation} \label{blowndown}
 \u_R(\x) := \bar{\Q}_\infty(R\x) \qquad \textrm{for } \x\in B_1.
\end{equation}
Due to Lemma~\ref{lemma:blowup}, $\u_R$ satisfies $|\u_R|\leq 1$ on~$B_1$,
$\u_R(\mathbf{0}) = \mathbf{0}$ and is a minimizer of the functional
\begin{equation} \label{GR}
 G_R(\u) = G_R(\u; \, B_1) := \int_{B_1} \left\{ \frac{\alpha}{p}\abs{\nabla\u}^p + \gamma R^p \left(1 - |\u|^2\right)^2\right\} \d V
\end{equation}
subject to its own boundary condition. Moreover, thanks to Lemma~\ref{lemma:bound_monotonicity},
the quantity $G_R(\u_R; B_1)\leq C$ is uniformly bounded, with respect to~$R$.
Therefore, up to extraction of a (non.relabelled) subsequence, the maps $\u_R$ converge $W^{1,p}$-weakly to a limit map~$\u_*\in W^{1, p}(B_1; \, S_0)$ that satisfies $|\u_*| = 1$ a.e. on~$B_1$. In other words, $\u_*$ takes values in the unit sphere of the $5$-dimensional Euclidean space~$S_0$. We denote this sphere by~$\S^4$.

\begin{lemma} \label{lemma:harmonic}
 For any~$1/2 < \rho < 1$, the map~${\u_*}_{|B_\rho}$ is $p$-minimizing harmonic: namely, for any map~$\u\in W^{1,p}(B_1; \, \S^4)$ such that $\u = \u_*$ a.e. on~$B_1\setminus B_\rho$, there holds
 \[
  \frac{1}{p} \int_{B_\rho} \abs{\nabla\u_*}^p \d V
  \leq \frac{1}{p} \int_{B_\rho} \abs{\nabla\u}^p \d V. 
 \]
 Moreover, ${\u_R}_{|B_\rho}$ converges $W^{1,p}$-strongly to~${\u_*}_{|B_\rho}$, as~$R\to+\infty$.
\end{lemma}

The proof of Lemma~\ref{lemma:harmonic} builds upon classical compactness arguments for harmonic maps which are due to Luckhaus~\cite{Luckhaus-PartialReg}, and is based on the following result,
which is an adaptation of Lemma~1 in~\cite{Luckhaus-PartialReg}. In contrast with the case considered in~\cite{Luckhaus-PartialReg}, we are dealing here with $2$-dimensional domains only; on the other hand, we have to include in our analysis the Ginzburg-Landau potential $(1 - |\u|^2)^2$, which is not present in~\cite{Luckhaus-PartialReg}. Similar results, for quadratic energies on three-dimensional domains, have been proven in~\cite{Canevari3D}. We denote
\[
 G_R(\u; \, \partial B_\rho) := \int_{\partial B_\rho}
 \left\{ \frac{\alpha}{p}\abs{\nabla\u}^p + \gamma R^p \left(1 - |\u|^2\right)^2\right\} \d S.
\]
 
\begin{lemma} \label{lemma:interpolation}
 There exist positive constants~$\delta$ and~$C$ with the following property. Let~$1/2 < \rho_0 < 1$, $0 < \lambda < 1/4$, $R > 4$ be given numbers, and set $\mu := \lambda + R^{-1}$. Let $\u, \v\in W^{1, p}(\partial B_{\rho_0}; \, S_0)$ be given maps that satisfy
 \begin{gather}
 \frac{1}{2} \leq |\u| \leq 2, \qquad |\v| = 1 \qquad \textrm{a.e. on } \partial B_{\rho_0} \label{hp:interpolation1} \\
 \left(\|\nabla\u\|_{L^p(\partial B_{\rho_0})}^{1/p} + \|\nabla\v\|_{L^p(\partial B_{\rho_0})}^{1/p}\right)\|\u - \v\|_{L^p(\partial B_{\rho_0})}^{1 - 1/p} \leq \delta. \label{hp:interpolation2}
 \end{gather}
 Then, there exists a  map~$\mathbf{w}\in W^{1,p}(B_{\rho_0}\setminus B_{\rho_0(1-\mu)}; \, S_0)$
 such that $\mathbf{w}(\x) = \u(\x)$ for a.e.~$\x\in \partial B_{\rho_0}$, $\mathbf{w}(\x) = \v((1-\mu)^{-1}\x)$ for a.e.~$\x\in \partial B_{\rho_0(1-\mu)}$, and
 \[
  \begin{split}
  &G_R(\mathbf{w}; \, B_{\rho_0}\setminus B_{\rho_0(1-\mu)}) \leq C\bigg( 
  \mu G_R(\u; \, \partial B_{\rho_0})
  + R^{-p^2/2+p-1} G_R^{p/2}(\u; \, \partial B_{\rho_0}) \\
  &\qquad\qquad + \lambda \int_{\partial B_{\rho_0}} \abs{\nabla\v}^p\d S 
  + \lambda^{1-p} \int_{\partial B_{\rho_0}} \abs{\u - \v}^p\d S\bigg) .
  \end{split}
 \]
\end{lemma}

We postpone the proof of Lemmas~\ref{lemma:harmonic} and~\ref{lemma:interpolation}, and conclude the proof of Theorem~\ref{th:no-isotropy} first. The last ingredient is the following regularity result for minimizing $p$-harmonic maps:

\begin{proposition} \label{prop:regularity-harmonic}
 Let~$\Omega$ be a smooth, bounded domain in~$\R^2$, and let $\u_*\colon\Omega\to\S^k$ be a $p$-minimizing harmonic map, with~$k\geq 2$, $1<p<+\infty$. Then, $\u_*\in C^{1, \alpha}_{\mathrm{loc}}(\Omega)$ for some~$\alpha\in (0, \, 1)$.
\end{proposition}

With the help the previous results, Theorem~\ref{th:no-isotropy} and 
Corollary~\ref{cor:biaxiality} follow easily.

\begin{proof}[Conclusion of the proof of Theorem~\ref{th:no-isotropy}]
 Thanks to Lemma~\ref{lemma:harmonic} and Proposition~\ref{prop:regularity-harmonic}, 
 we have strong convergence $\u_R\to\u_*$ in~$W^{1,p}_{\mathrm{loc}}(B_1)$, and moreover 
 $\u_*\in C^{1, \alpha}_{\mathrm{loc}}(B_1)$. Then, by adapting the uniform 
 convergence arguments above, we see that $\u_R\to\u_*$ locally uniformly in the
 open ball~$B_1$. But then we must have $\u_*(\mathbf{0}) = \mathbf{0}$, since
 $\u_R(\mathbf{0}) = \mathbf{0}$ for any~$R$. This yields the desired contradiction.
\end{proof}

\begin{proof}[of Corollary~\ref{cor:biaxiality}]
 For fixed~$\delta>0$, by Theorem~\ref{th:no-isotropy} we know that there 
 exists~$t_0 = t_0(\delta) > 0$ such that, when~$t\geq t_0$, there holds
 \begin{equation} \label{cor1}
  \Q_t(\x)\neq\mathbf{0} \qquad \textrm{for any } \x\in\Omega\colon
  \dist(\x, \, \partial\Omega)\geq\delta.
 \end{equation}
 By contradiction, suppose that the the boundary datum~$\Q_b\colon\partial\Omega\to\NN$
 is topologically non-trivial, that
 \begin{equation} \label{cor2}
  \lambda_{\max}(\Q_t(\x)) > \lambda_{\mathrm{mid}}(\Q_t(\x)) \qquad 
  \textrm{for any } \x\in\Omega\colon \dist(\x, \, \partial\Omega)\geq\delta
 \end{equation}
 and that
 \begin{equation} \label{cor3}
  \max_{\x\in\Omega} \beta^2(\Q_t(\x)) < 1.
 \end{equation}
 Thanks to~\eqref{cor1}--\eqref{cor3} and to the arguments in~\cite[Lemma~3.11]{Canevari2D},
 we can construct a continuous extension $\P\colon\Omega\to\NN$ of the boundary datum~$\Q_b$.
 This map has the form~$\P(\x) := s_+(\n(\x)\otimes\n(\x) - \mathbf{I}/3)$, where 
 $\n(\x)$ is a unit eigenvector associated with the leading eigenvalue 
 $\lambda_{\max}(\Q_t(\x))$. The existence of such an extension contradicts the 
 topological non-triviality of~$\Q_b$ and completes the proof.
\end{proof}

We now come back to the proof of the auxiliary results we used.

\begin{proof}[of Lemma~\ref{lemma:interpolation}]
 By a scaling argument, it suffices to prove the lemma in case~$\rho_0 = 1$.
 We work in polar coordinates~$(r, \, \theta)\in (0, \, 1)\times(0, \, 2\pi)$. For 
 $1 - R^{-1} < r < 1$, we define 
 \[
  t(r) := R(r - 1 + R^{-1})
 \]
 and
 \[
  \mathbf{w}(r, \, \theta) := \left(1 - t(r)\right) \frac{\u(\theta)}{|\u(\theta)|} + t(r) \u(\theta)
 \]
 i.e. we intepolate linearly, in the radial direction, between~$\u$ and~$\u/|\u|$ to define~$\mathbf{w}$ on the annulus~$B_1\setminus B_{1 - R^{-1}}$. A straightforward computation gives
 \begin{equation} \label{interp1}
  \abs{\partial_r\mathbf{w}} = R \abs{1 - |\u|} \leq R\abs{1 - |\u|^2}, \quad
  \abs{\partial_\theta\mathbf{w}} \leq \abs{\partial_\theta\u} +  \abs{\partial_\theta\left(\frac{\u}{|\u|}\right)}
  \leq C \abs{\partial_\theta\u}
 \end{equation}
 because $|\u|\geq 1/2$ by assumption~\eqref{hp:interpolation1} and the map $\mathbf{y}\mapsto \mathbf{y}/\mathbf{|y|}$ is Lipschitz continuous on the set~$\{|\mathbf{y}|\geq 1/2\}$. Moroever,
 since $|\u|\leq 2$ by~\eqref{hp:interpolation1}, we also have~$|\mathbf{w}|\leq 2$ and hence
 \begin{equation} \label{interp2}
  \abs{1 - |\mathbf{w}|^2} \leq 3 \abs{1 - |\mathbf{w}|} = 3t\abs{1 - |\u|} \leq 3\abs{1 - |\u|^2}.
 \end{equation}
 Since~$p < 2$, the H\"older inequality implies
 \begin{equation} \label{interp3}
  \int_{\partial B_1} \abs{1 - |\u|^2}^p \d S
  \leq C \left( \int_{\partial B_1} \abs{1 - |\u|^2}^2\d S\right)^{p/2} \leq C R^{-p^2/2} G_R^{p/2}(\u; \, \partial B_1)
  \end{equation}
  for some constant~$C$ that only depends on~$p$ and~$\gamma$. By combining~\eqref{interp1},
  \eqref{interp2} and~\eqref{interp3}, and integrating over the annulus $(\rho, \, \theta)\in (1-R^{-1}, \, 1)\times (0, \, 2\pi)$, we obtain that
  \begin{equation} \label{interp4}
   \begin{split}
   G_R(\mathbf{w}; \, B_1\setminus B_{1 - R^{-1}}) &\leq C R^{-1} G_R(\u; \, \partial B_1)  
   + CR^{p-1}\int_{\partial B_1} \abs{1 - |\u|^2}^p \d S \\
   &\leq C R^{-1} G_R(\u; \, \partial B_1) + C R^{-p^2/2+p-1} G_R^{p/2}(\u; \, \partial B_1).
   \end{split}
  \end{equation}
  Note that the Jacobian factor arising from the change of variable, as well as the extra powers of~$r$ arising in the expression of~$|\nabla\mathbf{w}|$, are uniformly bounded because we have assumed that $R> 4$, and hence~$3/4<r<1$.
  
  Now, for $1 - \mu < r < 1 - R^{-1}$ (where we have set $\mu := \lambda + R^{-1}$) we define
  \[
  s(r) := \lambda^{-1}(r - 1 + \mu)
  \]
  and
  \[
   \tilde{\mathbf{w}}(r, \, \theta) := \left(1 - s(r)\right)\v(\theta) + s(r) \frac{\u(\theta)}{|\u(\theta)|} 
  \]
  i.e. we intepolate linearly, in the radial direction, between~$\u/|\u|$ and~$\v$.
  Using again the Lipschitz continuity of $\mathbf{y}\mapsto \mathbf{y}/\mathbf{|y|}$ on the set~$\{|\mathbf{y}|\geq 1/2\}$ and the fact that $|\u|\geq 1/2$, $|\v| = 1$ by assumption~\eqref{hp:interpolation1}, we compute that
  \[
   \abs{\partial_r\mathbf{w}} \leq  \lambda^{-1}\abs{\frac{\u}{|\u|} - \v} \leq \lambda^{-1}\abs{\u - \v}, \qquad
  \abs{\partial_\theta\mathbf{w}} \leq \abs{\partial_\theta\u} +  \abs{\partial_\theta\v}
  \]
  so that 
  \begin{equation} \label{interp5}
   \int_{B_{1-R^{-1}}\setminus B_{1-\mu}} \abs{\nabla\tilde{\mathbf{w}}}^p
   \leq C \left(\lambda\int_{\partial B_1} (\abs{\nabla\u}^p +\abs{\nabla\v}^p)\d S+\lambda^{1 - p}\int_{\partial B_1} \abs{\u - \v}^p\d S \right) \!.
  \end{equation}
  Now, the $1$-dimensional Gagliardo Nirenberg inequality 
  \[
   \|\u - \v\|_{L^\infty(\partial B_1)} \leq C \|\nabla\u - \nabla\v\|_{L^p(\partial B_1)}^{1/p} \|\u - \v\|_{L^p(\partial B_1)}^{1 - 1/p},
  \]
  together with our assumption~\eqref{hp:interpolation2}, implies that $|\u - \v|\leq C\delta$
  for some constant~$C$ that only depends on~$p$. Therefore, we have
  \[
   |\tilde{\mathbf{w}}| \geq 1 - \abs{\frac{\u}{|\u|} - \v} \geq 1 - C\abs{\u - \v} \geq 1 - C\delta
  \]
  and, by taking~$\delta$ small enough, we can make sure that $|\tilde{\mathbf{w}}|\geq 1/2$. We can then define $\mathbf{w} := \tilde{\mathbf{w}}/|\tilde{\mathbf{w}}|$ on the annulus~$B_{1-R^{-1}}\setminus B_{1-\mu}$, and the lemma now follows thanks to~\eqref{interp4} and~\eqref{interp5}.  
\end{proof}

\begin{proof}[of Lemma~\ref{lemma:harmonic}]
 Let~$1/2 < \rho < 1$ be fixed. By Fatou lemma, we have
 \[
  \int_{\rho}^1\liminf_{R\to+\infty} G_R(\u_R; \, \partial B_t) \, \d t \leq 
  \liminf_{R\to+\infty} G_R(\u_R; \, B_1\setminus B_\rho) \leq C
 \]
 and hence, there exists~$\rho_0\in (\rho, \, 1)$ such that, by possibly taking a (non-relabelled) subsequence~$R\to+\infty$, there holds
 \begin{equation} \label{harmonic1}
  G_R(\u_R; \, \partial B_{\rho_0}) \leq \frac{C}{1-\rho} =: C_\rho, \qquad 
  \lim_{R\to +\infty}\int_{\partial B_{\rho_0}} \abs{\u_R - \u_*}^p \d S= 0.
 \end{equation}
 Let us define 
 \[
  \lambda_R := \begin{cases}
  \left(\int_{\partial B_{\rho_0}} \abs{\u_R - \u_*}^p\d S\right)^{1/p} &\textrm{if } 
  \int_{\partial B_{\rho_0}} \abs{\u_R - \u_*}^p\d S\neq 0 \\
  1/R & \textrm{otherwise}.
  \end{cases}
 \]
 Then, thanks to \eqref{harmonic1}, $\lambda_R$ is a positive sequence that satisfies
 \begin{equation} \label{harmonic2}
  \lambda_R\to 0, \quad \lambda_R^{1-p}\int_{\partial B_{\rho_0}} \abs{\u_R - \u_*}^p\d S\to 0
  \qquad \textrm{as } R\to +\infty.
 \end{equation}
  We aim to apply Lemma~\ref{lemma:interpolation} with~$\v = \u_*$. Thanks to the compact Sobolev embedding $W^{1,p}(\partial B_{\rho_0})\hookrightarrow C^0(\partial B_{\rho_0})$ and to~\eqref{harmonic1}, we have $\u_R\to \u_*$ uniformly on~$\partial B_{\rho_0}$ and~$|\u_*| = 1$ a.e. on~$\partial B_{\rho_0}$, so the assumption~\eqref{hp:interpolation1} is satisfied for~$R$ large enough. Moreover, \eqref{hp:interpolation2} is also satisfied for~$R$ large enough, due to~\eqref{harmonic1}. Therefore, we can apply the lemma. Letting $\mu_R := \lambda_R + R^{-1}$, we find a map~$\mathbf{w}_R\in W^{1,p}(B_{\rho_0}\setminus B_{\rho_0(1-\mu_R)}; \, S_0)$ such that 
 \begin{gather*}
  \mathbf{w}_R(\x) = \u_R(\x) \qquad \textrm{for a.e. } 
     \x\in \partial B_{\rho_0} \nonumber \\
  \mathbf{w}_R(\x) = \u_*\left(\frac{\x}{1-\mu_R}\right) \qquad 
     \textrm{for a.e. } \x\in \partial B_{\rho_0(1-\mu_R)}
 \end{gather*}
 and
 \begin{equation*}
  \begin{split}
  &G_R(\mathbf{w}_R; \, B_{\rho_0}\setminus B_{\rho_0(1-\mu_R)}) \leq C\bigg( 
  \mu_R  G_R(\u_R; \, \partial B_{\rho_0}) 
  + R^{-p^2/2+p-1} G_R^{p/2}(\u_R; \, \partial B_{\rho_0}) \\ &\qquad\qquad
  + \lambda_R \int_{\partial B_{\rho_0}} \abs{\nabla\u_*}^p \d S 
  + \lambda_R^{1-p} \int_{\partial B_{\rho_0}} \abs{\u_R - \u_*}^p\d S\bigg)
  \end{split}
 \end{equation*}
 so, thanks to~\eqref{harmonic1} and~\eqref{harmonic2},
 \begin{equation} \label{harmonic3}
  \lim_{R\to +\infty} G_R(\mathbf{w}_R; \, B_{\rho_0}\setminus B_{\rho_0(1-\mu_R)}) = 0.
 \end{equation}
 Now, let~$\u\in W^{1,p}(B_1; \, S_0)$ be a function such that $\u = \u_*$ a.e. on~$B_1\setminus B_\rho$. We define
 \[
  \v_R(\x) := \begin{cases}
   \u_R(\x)                            & \textrm{for } \x\in B_1\setminus B_{\rho_0} \\
   \mathbf{w}_R(\x)                         & \textrm{for } \x\in B_{\rho_0}\setminus B_{\rho_0(1 -\mu_R)} \\
   \u\left(\dfrac{\x}{1 - \mu_R}\right) & \textrm{for } B_{\rho_0(1 -\mu_R)}.
  \end{cases}
 \]
 The map~$\v_R$ belongs to~$W^{1,p}(B_1; \, S_0)$ and agrees with~$\u_R$ on~$\partial B_{\rho_0}$; moreover, due to~\eqref{harmonic3}, we have
 \begin{equation} \label{harmonic4}
  G_R(\v_R; \, B_{\rho_0}) = \frac{\alpha}{p} (1 - \mu_R)^{2-p} \int_{B_{\rho_0}} \abs{\nabla\u}^p + \mathrm{o}(1) \qquad\textrm{as } R\to+\infty.
 \end{equation}
 Thanks to the weak convergence $\u_R\rightharpoonup\u_*$ in~$W^{1,p}$, the minimality of~$\u_R$, and~\eqref{harmonic4}, we obtain that
 \[
 \frac{\alpha}{p}\int_{B_{\rho_0}} \abs{\nabla\u_*}^p \leq \liminf_{R\to+\infty} G_R(\u_R; \, B_{\rho_0}) \leq \liminf_{R\to+\infty} G_R(\v_R; \, B_{\rho_0}) = \frac{\alpha}{p}\int_{B_{\rho_0}} \abs{\nabla\u}^p ,
 \]
 so~${\u_*}_{|B_{\rho_0}}$ is $p$-minimizing harmonic. By taking~$\u = \u_*$, the same agument also shows that $\|\nabla\u_R\|_{L^p(B_{\rho_0})}\to\|\nabla\u_*\|_{L^p(B_{\rho_0})}$ as~$R\to+\infty$, whence we deduce the strong convergence ${\u_R}_{|B_{\rho_0}}\to{\u_*}_{|B_{\rho_0}}$ in~$W^{1,p}$.
\end{proof}

Finally, we give the proof of Proposition~\ref{prop:regularity-harmonic}. In case~$p=2$, the result is known by the work of H\'elein \cite{Helein1991a}. In case~$p>2$, the proposition follows by the results of Hardt and Lin~\cite[Corollary~2.6 and Theorem~3.1]{HardtLin}. In case~$1<p<2$, it suffices to prove the following lemma:

\begin{lemma} \label{lemma:MTM}
 Let~$B_1$ be the unit disk in~$\R^2$, let $k\geq 2$, $1<p<2$, and let $\mathbf{w}\colon B_1\to\S^{k}$ be a $p$-minimizing harmonic map that is homogeneous of degree~$0$, i.e. it satisfies $\mathbf{w}(\x) = \mathbf{w}(\x/|\x|)$ for a.e.~$\x\in B_1$. Then, $\mathbf{w}$ is constant.
\end{lemma}

Once Lemma~\ref{lemma:MTM} is proven, Proposition~\ref{prop:regularity-harmonic} follows by Federer's ``dimension reduction'' argument (see e.g. \cite[Theorem~4.5]{HardtLin} or \cite[Theorem~IV]{SchoenUhlenbeck}).

\begin{proof}[of Lemma~\ref{lemma:MTM}]
 Since~$\mathbf{w}$ is homogeneous of degree~$0$, by working in polar coordinates we can identify~$\mathbf{w}$ with a periodic function of one scalar variable~$\theta\in(0, \, 2\pi)$. The $p$-harmonic map equation then writes
 \begin{equation} \label{MTM0}
 \frac{\d}{\d\theta} \left(\abs{\mathbf{w}^\prime}^{p-2}\mathbf{w}^\prime\right) = \abs{\mathbf{w}^\prime}^p\mathbf{w}
 \end{equation}
 where $\mathbf{w}^\prime := \frac{\d}{\d\theta}\mathbf{w}$. By taking the scalar product of both sides of the equation with~$\mathbf{w}^\prime$, and using the identity $\mathbf{w}\cdot\mathbf{w}^\prime = 0$ (which follows by differentiating $|\mathbf{w}|=1$), we obtain
 \[
  \begin{split}
 0 &= \frac{\d}{\d\theta} \left(\abs{\mathbf{w}^\prime}^{p-2}\mathbf{w}^\prime\right)\cdot\mathbf{w}^\prime = \frac{\d}{\d\theta}\abs{\mathbf{w}^\prime}^p - \abs{\mathbf{w}^\prime}^{p-2}\mathbf{w}^\prime \cdot\mathbf{w}^{\prime\prime} \\
 &= (p-1) \abs{\mathbf{w}^\prime}^{p-2}\mathbf{w}^\prime \cdot\mathbf{w}^{\prime\prime} = \frac{p-1}{p}\frac{\d}{\d\theta} \left(\abs{\mathbf{w}^\prime}^p\right).
 \end{split}
 \]
 Then, $|\mathbf{w}^\prime|$ is constant, and hence~\eqref{MTM0} reduces to the harmonic map equation:
 \[
  \mathbf{w}^{\prime\prime} = \abs{\mathbf{w}^\prime}^2\mathbf{w}.
 \]
 Therefore, $\mathbf{w}$ must parametrize a closed geodesic in~$\S^k$, that is, a great circle. Up to rotations, we might assume that
 \[
  \mathbf{w}(\theta) = \left(\cos(j\theta), \, \sin(j\theta), 0, \, \ldots, \, 0\right)
 \]
 for some integer~$j$. We fix a function~$\psi\colon [0, \, 1]\to\R$ with~$\psi(1) = 0$ and, for a small number~$t\in\R$, we consider the family of maps~$\mathbf{w}_t\colon B_1\to\S^k$ given in polar coordinates by
 \[
  \mathbf{w}_t(r, \, \theta) :=  \left(\cos(t\psi(r))\cos(j\theta), \, \cos(t\psi(r))\sin(j\theta), \, \sin(t\psi(r)), \, 0, \, \ldots, \, 0\right) \!.
 \]
 We have~$\mathbf{w}_0=\mathbf{w}$ and~$\mathbf{w}_t = \mathbf{w}$ on~$\partial B_1$ for any~$t$, therefore, the minimimality of $\mathbf{w}$ implies that
 \begin{equation} \label{MTM1}
  I := \frac{\d^2}{\d t^2}_{|t=0} \int_{B_1} \abs{\nabla\mathbf{w}_t}^p \geq 0.
 \end{equation}
 On the other hand, we can explicitely compute~$I$. Indeed, we have
 \begin{gather*}
  \abs{\nabla\mathbf{w}_t}^2 = t^2{\psi^\prime}^2 + \frac{j^2}{r^2}\cos^2(t\psi) =
  \frac{j^2}{r^2} + t^2\left({\psi^\prime}^2 - \frac{j^2}{r^2}\psi^2\right) + \mathrm{o}(t^2)\\
  \abs{\nabla\mathbf{w}_t}^p =
   \frac{j^p}{r^p} + \frac{p}{2}\frac{j^{p-2}}{r^{p-2}} t^2 \left({\psi^\prime}^2 - \frac{j^2}{r^2}\psi^2\right) + \mathrm{o}(t^2) \\
   I = 2\pi p j^{p-2} \int_0^1 \left(r^{3-p}{\psi^\prime}^2 - j^2r^{1-p}\psi^2\right) \d r
 \end{gather*}
 We first consider the case~$|j|=1$, and choose $\psi(r) := 1 - r^\alpha$ for some parameter~$\alpha>p/2-1$ to be specified later. We have $\psi^\prime(r) = -\alpha r^{\alpha - 1}$ and
 \[
  \begin{split}
   \frac{I}{2\pi p} &= \frac{\alpha^2-1}{2 - p + 2\alpha} + \frac{2}{2 - p + \alpha} - \frac{1}{2 - p} \\
   &= \frac{\alpha^2 (\alpha(2 - p) + p^2 - 4 p + 2)}{(2 - p + 2\alpha)(2 - p + \alpha)(2 - p)}
  \end{split}
 \]
 Since $p^2 - 4 p + 2<0$ for $1 < p <2$, by taking $\alpha = \alpha(p)>0$ small enough we can make sure that~$I<0$, which contradicts~\eqref{MTM1}. In case~$|j|>1$ we have, a fortiori, $I<0$ with the same choice of~$\psi$. Thus, we must have $j=0$ that is, $\mathbf{w}$ is constant.
\end{proof}

\begin{acknowledgements}
 A. M. would like to thank John Ball for suggesting this problem to her when she was a postdoctoral researcher at OxPDE.
 Part of this work was carried out when the authors were visiting
 the International Centre for Mathematical Sciences (ICMS) in Edinburgh (UK),
 supported by the Research-in-Groups program.
 The authors would like to thank the ICMS for its hospitality.
 G. C.'s  research  was supported by the Basque Government through the BERC 2018-2021 program
 and by the Spanish Ministry of Economy and Competitiveness: MTM2017-82184-R.
 A.M. is supported by an EPSRC Career Acceleration Fellowship EP/J001686/1 and EP/J001686/2 and an OCIAM
 Visiting Fellowship, the Keble Advanced Studies Centre.
 B.S.'s research was supported by the Project: Variational Advanced TEchniques for compleX MATErials (VATEXMATE) of University Federico II of Naples.
 B.S. would like to thank the OxPDE center whose hospitality in Michaelmas term 2015 and 2016 made it possible to interact with G.C. and A.M. and with the research group on Liquid Crystals.
\end{acknowledgements}


\newcommand{\noop}[1]{}

\end{document}